\documentclass[11pt]{amsart}
\usepackage{amsmath}
\usepackage{bm}
\usepackage{dsfont}
\usepackage{mathrsfs}
\usepackage{color}
\usepackage[T1, OT1]{fontenc}
\usepackage{hyperref}
\usepackage{subfig}
\usepackage{float}
\usepackage{thmtools,thm-restate}
\usepackage{setspace}
\allowdisplaybreaks
\DeclareSymbolFont{bbold}{U}{bbold}{m}{n}
\DeclareSymbolFontAlphabet{\mathbbold}{bbold}
\newcommand{\ind}{\mathbbold{1}}

\numberwithin{equation}{section}

\theoremstyle{plain} \newtheorem{theorem}{Theorem}[section]
\theoremstyle{plain} \newtheorem{lemma}[theorem]{Lemma}
\theoremstyle{plain} \newtheorem{assumption}[theorem]{Assumption}
\theoremstyle{plain} \newtheorem{corollary}[theorem]{Corollary}
\theoremstyle{plain} \newtheorem{proposition}[theorem]{Proposition}
\theoremstyle{plain} \newtheorem{remark}[theorem]{Remark}
\theoremstyle{definition} 
\theoremstyle{plain} \newtheorem{definition}[theorem]{Definition}

\theoremstyle{plain} 

\newcommand{ \R}{ \mathbb R }
\newcommand{ \C}{ \mathbb C }
\newcommand{\E}{ \mathbb E}
\newcommand{\N}{ \mathbb N}

\renewcommand{\Pr}{ \mathbb P}

\newcommand{\Qr}{ \mathbb Q}

\newcommand{\cD}{\mathcal{D}}

\newcommand {\vp}{\tilde\varphi}
\newcommand {\tth}{\tilde\theta}
\newcommand {\ve}{\varepsilon}

\begin{document}
\title[Quasistationary distributions ]{Quasistationary distributions for one-dimensional diffusions with singular boundary points}

\author[A. Hening]{Alexandru Hening }
\address{Department of Mathematics\\
Tufts University\\
Bromfield-Pearson Hall\\
503 Boston Avenue\\
Medford, MA 02155\\
United States
}
\email{alexandru.hening@tufts.edu}
\author[M. Kolb]{Martin Kolb}
\address{Department of Mathematics \\
 University of Paderborn \\
 Warburger Str. 100 \\
 33098 Paderborn\\
 Germany}
 \email{kolb@math.uni-paderborn.de}

\renewcommand{\thefootnote}{\fnsymbol{footnote}}
\footnotetext{\emph{AMS subject classification} Primary 58J05, 60J60; Secondary 35J10, 35P05, 47F05.}
\renewcommand{\thefootnote}{\arabic{footnote}}

\renewcommand{\thefootnote}{\fnsymbol{footnote}}
\footnotetext{ \textbf{Keywords.} One-dimensional diffusion, quasistationary distribution, Yaglom limit, $Q$ process. }
\renewcommand{\thefootnote}{\arabic{footnote}}

\begin{abstract}
In the present work we characterize the existence of quasistationary distributions
for diffusions on $(0,\infty)$ allowing singular behavior at $0$ and $\infty$. If absorption at
0 is certain, we show that there exists a quasistationary distribution as soon as the spectrum
of the generator is strictly positive. This complements results of Collet et al. (Ann. Probab. 2009) and
Kolb and Steinsaltz (Ann. Probab. 2012) for $0$ being a regular boundary point and extends results by Collet et al. (Ann. Probab. 2009) on singular diffusions.

\end{abstract}
\maketitle

\tableofcontents

\section{Introduction and main results}

Throughout this paper we consider the one-dimensional diffusion $(X_t)_{t\geq 0}$ that is the solution to the stochastic differential equation (SDE)
\begin{equation}\label{e:gen_SDE}
dX_t = dB_t -b(X_t)\,dt,
\end{equation}
where $(B_t)_{t\geq 0}$ is a standard one-dimensional Brownian motion and $b:(0,\infty)\to\R$ is the drift.

We make the following assumption throughout the paper.
\begin{assumption}\label{a:drift}
The drift is continuously differentiable on $(0,\infty)$, that is $b\in C^1((0,\infty))$.
\end{assumption}

The first hitting time of $0$ by $(X_t)_{t\geq 0}$ is denoted by
\begin{equation}\label{e:T0}
  T_0:= \inf \{t\geq 0: X_t=0\}
\end{equation}
We are interested in the existence of quasistationary distributions of $(X_t)_{t\geq 0}$ with absorption at $0$. These are invariant measures of the evolution \textit{conditioned on not hitting zero}.

\begin{definition}\label{d:qsd_1}
A probability measure $\nu$ on $(0,\infty)$ satisfying
\[
\Pr^\nu (X_t\in A~|~T_0>t) = \nu(A)
\]
for all Borel sets $A\subset (0,\infty)$ is called a quasistationary distribution of $(X_t)_{t\geq 0}$.
\end{definition}
The probability measure $\Pr^\nu$ expresses the fact that $X_0$ has initial distribution $\nu$. Note that the event $\{T_0>t\}$ can be interpreted as survival until at least time $t$. Studying the existence, uniqueness, and the properties of quasistationary distributions for one-dimensional diffusions as well as for Markov chains are challenging tasks. In recent years, these problems have attracted the interest of both probabilists and biologists. Quasistationary distributions have been used in mathematical biology (see \cite{C09}), demography (see \cite{ES04}) and in models of neutron transport (see \cite{CV14}). Some recent developments and a description of the role of quasistationary distributions in models derived from ecology and population dynamics are described in the excellent survey article \cite{MV12}.

One possible approach to analyzing the existence question for quasistationary distributions (QSD) is to look at the following limit.

\begin{definition}\label{d:Yaglom}
Let $\nu$ be a probability measure on $(0,\infty)$. The Yaglom limit of $(X_t)_{t\geq 0}$ is defined (if it exists) by
\begin{equation}\label{e:yaglom}
\lim_{t\to \infty}\Pr^\nu(X_t\in\cdot ~|~ T_0>t).
\end{equation}
\end{definition}

It is easy to see that if a non-trivial Yaglom limit exists, the limit is a quasistationary distribution (see for example \cite{C09}). Therefore, in order to analyse the existence problem for QSDs it is reasonable to look at Yaglom limits. It will turn out that in many situations the Yaglom limit is either non-trivial or trivial in the following sense.

\begin{definition}
We say that $(X_t)_{t\geq 0}$, started with the initial distribution $\nu$, \textit{converges to the quasistationary distribution $u:(0,\infty)\to\R$}, if $u$ is a non-negative function and for all Borel measurable sets $A\subset (0,\infty)$
\begin{equation}\label{e:conv_nontrivial}
\lim_{t\to \infty} \Pr^\nu (X_t\in A~|~T_0>t) = \frac{\int_A u(x)\rho(dx)}{\int_0^\infty u(x)\rho(dx)},
\end{equation}
where $\rho(dx)=\rho(x)dx= \exp\left(-\int_1^x 2b(s)\,ds\right) dx$ is the speed measure of the diffusion.

We say that $(X_t)_{t\geq 0}$, started with the initial distribution $\nu$, \textit{escapes to infinity} if for every $a\in(0,\infty)$
\[
\lim_{t\to\infty} \Pr^\nu(X_t\in [0,a)~|~T_0>t)=0.
\]
\end{definition}

In many cases the convergence to a quasistationary distribution is independent of the initial distribution $\nu$. This follows from our characterization of the convergence to a quasistationary distribution via the strict negativity of the exponential rate of decay of $\Pr^\nu(T_0>t)$ and Harnack's inequality.

The main contribution of this work is a better understanding of the Yaglom limit for one-dimensional diffusions. While in \cite{KS12} the existence problem has been solved for regular drift functions $b\in C^1([0,\infty))$ and internal killing, the present work is concerned with drift functions $b\in C^1((0,\infty))$ having a \textit{singularity} at $0$ -- we will in particular look at variants of the Bessel process corresponding to $b(x)=\frac{a}{x}$.

An interesting application of the singular case is the study of generalized Feller diffusions of the type
\begin{equation}\label{e:feller}
dZ_t = h(Z_t)dt - \sqrt{Z_t}dB_t,
\end{equation}
where $h$ is a smooth enough function. Such stochastic differential equations arise as scaling limits of certain discrete population models as shown in \cite{C09}. If one lets $Y_t:=2\sqrt{Z_t}$ then an easy application of It{\^o}'s Lemma together with \eqref{e:feller} yields
\begin{equation}\label{e:Feller2}
dY_t = -dB_t -\frac{1}{Y_t}\left(\frac{1}{2}-2h\left(\frac{Y_t^2}{4}\right)\right)dt.
\end{equation}
The SDE for $(Y_t)_{t\geq 0}$ is therefore of the form \eqref{e:gen_SDE} and the drift function has a singularity at zero.

The main arguments of the present work are based on analytic extensions of results given in \cite{M61} to a wider class of diffusions. Using these analytic results one can understand the convergence to quasistationary distributions on compact sets $A\subset (0,\infty)$. From this, with the help of an elegant argument from \cite{ES07}, we can deduce that $(X_t)_{t\geq 0}$ either converges to a quasistationary distribution or escapes to infinity. In many instances this dichotomy can be reduced to understanding the exponential rate of decay of $\Pr^\nu(T_0>t)$. This leads to a characterization of the convergence of $(X_t)_{t\geq 0}$ via the positivity of the bottom of the spectrum of the diffusion operator $L$ corresponding to
\[
\tau = -\frac{1}{2}\frac{d^2}{dx^2} + b(x)\frac{d}{dx}.
\]
Throughout this paper we denote the bottom of the spectrum by
\begin{equation}\label{e:bottom}
\lambda_0 :=\inf \text{spec} (L).
\end{equation}

After tackling the question of the existence of quasistationary distributions, one may also want to consider the $Q$-process. This is the process that arises from the original diffuson when conditioned on non-extinction

\[
\lim_{t\to\infty} \Pr^x(X_{\bullet}\in\cdot~|~T_0>t).
\]

Our results can be used in conjunction with classical techniques to characterize the existence of the $Q$-process for a large class of diffusions.

\subsection{Main results}Let us start by recalling Feller's boundary classification for one-dimensional diffusions. \textit{Accessible} boundary points (points which can be reached in finite time with positive probability) are divided into \textit{regular} and \textit{exit} points, whereas \textit{inaccessible} points (points which cannot be reached in finite time with positive probability) are divided into \textit{entrance} and \textit{natural} points.

More specifically, we have the following.

\begin{definition}
Let $c\in(0,\infty)$ be given and remember that the density of the speed measure of \eqref{e:gen_SDE} is
\[
\rho(x)= \exp\left(-\int_1^x2b(s)\,ds\right),~~ x\in(0,\infty).
\]
The point $\infty$ is \emph{accessible}, if $\int_c^\infty\rho(x)^{-1}\int_c^x\rho(y)\,dy\,dx<\infty$,
and otherwise \emph{inaccessible}. If $\infty$ is an accessible boundary point, then it is called \emph{regular} if and only if $\int_c^\infty\rho(x)\int_c^x\rho(y)^{-1}\,dy\,dx<\infty.$ If $\infty$ is accesible and $\int_c^\infty\rho(x)\int_c^x\rho(y)^{-1}\,dy\,dx=\infty$
then $\infty$ is called an \emph{exit boundary}. If $\infty$ is inaccesible, then it is an \emph{entrance boundary}, if and only if $\int_c^\infty\rho(x)\int_c^x\rho(y)^{-1}\,dy\,dx<\infty$. If $\infty$ is inaccesible and $\int_c^\infty\rho(x)\int_c^x\rho(y)^{-1}\,dy\,dx=\infty$ then $\infty$ is called \emph{natural}. A similar classification holds for the boundary point $0$.

\end{definition}

We are interested in conditioning on not hitting zero. Therefore, only the case when $0$ is accessible needs to be considered because otherwise the conditioning would be redundant. The Yaglom limit in the regular case, first discussed in \cite{CMSM95} (a gap in the arguments was pointed out in \cite{ES07}), was fully classified in \cite{KS12}. As a result, in this paper we only consider the singular case when $0$ is an exit boundary.

The boundary at $\infty$ is generally assumed to be inaccessible. This is a natural assumption as far as biological applications (see \cite{C09}) are concerned since populations do not blow up to $\infty$ in nature.

As already observed in \cite{CMSM95} for $0$ regular and $\infty$ natural, it is crucial if certain absorption in the following sense holds or not.

\begin{definition}\label{d:absorption}
We say that absorption is \textit{certain} if $\Pr^x(T_0<\infty)=1$ for any starting point $x\in (0,\infty)$.
\end{definition}

Equivalently, certain absorption can be reformulated by the finiteness of the speed measure
\[
\Pr^x(T_0<\infty)<1 \iff \int_1^\infty \rho(s)^{-1}\,ds <\infty.
\]This quantifies the minimum strength of the drift towards $0$ one needs in order to obtain certain absorption.
The cases in which absorption is not certain are easy to treat since one can show that for $a>0$,
\[
\Pr^\nu(X_t\in[0,a), T_0>t)\downarrow 0 ~\text{as}~t\to\infty.
\]
Combining this with the boundedness from below of $\Pr^\nu(T_0>t)$, implies that
\[
\lim_{t\to\infty}\Pr^\nu(X_t\in A~|~T_0>t) = \lim_{t\to\infty}\frac{\Pr^\nu(X_t\in A, T_0>t)}{\Pr^\nu(T_0>t)}=0.
\]
This shows that if absorption is not certain the process $(X_t)_{t\geq 0}$ escapes to infinity.

To get around uncertain absorption, one can condition on the event of eventual absorption. This can be done using an $h$-transform. The $h$-transformed process will be have certain absorption and so one can then apply our main results. This is discussed in Remark \ref{r:doob}.

 For the case when $\infty$ is a natural boundary, we need a regularity condition on the semigroup generated by the diffusion.
\begin{assumption} \label{a:1}
Let $z\in (0,\infty)$ and $t>0$. Then
\[
e^{-tL}\mathbf{1}_{[0,z]}\in L^2((0,\infty),\rho)
\]
and
\[
\lim_{\varepsilon\to 0} e^{-tL}\mathbf{1}_{[\varepsilon,z]} =e^{-tL}\mathbf{1}_{[0,z]}
\]
in $L^2((0,\infty),\rho):=\left\{f:(0,\infty)\to\R, f~\text{measurable and}~\int_0^\infty|f(x)|^2\,\rho(dx)<\infty\right\}$.

\end{assumption}
The next result tells us under what conditions, when $\infty$ is a natural boundary point, we have the convergence of $(X_t)_{t\geq 0}$ to a quasistationary distribution. As before, $\lambda_0$ denotes the bottom of the spectrum of $L$.
\begin{restatable}{theorem}{natural}\label{t:exit_natural}
Suppose that $0$ is an exit boundary, $\infty$ is a natural boundary and Assumptions \ref{a:drift} and \ref{a:1} hold. We have the following classification.
\begin{itemize}
\item[i)] If absorption is not certain and $\lambda_0>0$, then $(X_t)_{t\geq 0}$ escapes to infinity exponentially fast.
\item[ii)] If absorption is certain and $\lambda_0>0$, then $(X_t)_{t\geq 0}$ converges to a quasistationary distribution. This quasistationary distribution is the unique Yaglom limit and it attracts all the compactly supported initial distributions of $X_0$.
\end{itemize}
\end{restatable}
\begin{remark}
We do not claim that Assumption \ref{a:1} is necessary but we were not able to prove the result without it.
\end{remark}
The following result gives sufficient conditions on the drift $b$ guaranteeing that the spectrum is positive.
\begin{restatable}{proposition}{pos}\label{p:positivity_spectrum}
Assume that $0$ is an exit boundary and that absorption is certain. The bottom of the spectrum $\lambda_0$ is strictly positive if and only if for some $a\in (0,\infty)$
\[
A(b,a):= \sup_{x>a}\left(\int_a^x \rho(y)^{-1}\,dy\right)\left(\int_x^\infty\rho(y)\,dy\right)<\infty
\]
In particular the strict positivity of $\lambda_0$ implies $\int_1^\infty \rho(y)\,dy<\infty$.
\end{restatable}

\begin{remark}
In the proof of Theorem of \ref{t:exit_natural} we show that if absorption is not certain, then $\Pr^\nu(X_t\in[0,a) ~|~ T_0>t)$ decays at least exponentially fast with rate $\lambda_0$. When $0$ is regular and $\infty$ natural this result was proved in Theorem 4 of \cite{MSM01}.
\end{remark}

Theorem \ref{t:exit_natural} immediately implies the following Corollary.
\begin{corollary}
If $0$ is an exit boundary, $\infty$ is a natural boundary, Assumption \ref{a:1} is fulfilled and absorption is certain, then there exists a quasistationary distribution if and only if $\lambda_0>0$.

If $0$ is an exit boundary, $\infty$ is a natural boundary, absorption is not certain and $\lambda_0>0$ then there are no quasistationary distributions.

\end{corollary}
\begin{proof}
If absorption is certain and $\lambda_0>0$ then Theorem \ref{t:exit_natural} gives the existence of a quasistationary distribution. 

If instead, one assumes that there exists a quasistationary distribution $\nu$ then it is known that $T_0$ must have an exponential moment (see Proposition 3 from \cite{MV12}). This then implies that $\lambda_0>0$.
\end{proof}

Intuitively, convergence to a quasistationary distribution should become easier for entrance boundaries at $\infty$ since this implies there is a strong drift towards $0$. Hence, escape to infinity gets more difficult and convergence to a quasistationary distribution more likely. The next theorem shows that this is true even in the absence of Assumption \ref{a:1}.

\begin{restatable}{theorem}{entrance}\label{t:exit_entrance}
Suppose Assumption \ref{a:drift} holds. If $0$ is an exit boundary and $\infty$ is an entrance boundary, then $\lambda_0>0$ and $(X_t)_{t\geq 0}$ converges to the unique quasistationary distribution.
\end{restatable}

Note that we do not have to distinguish between certain and uncertain absorption. For $\infty$ being an entrance boundary, absorption at $0$ is certain. In \cite{C09}, the authors have shown that $(X_t)_{t\geq 0}$ converges to a quasistationary distribution under a set of assumptions which ensure that the spectrum of $L$ is discrete and that the lowest eigenfunction is integrable. The fact that this is always true is new and completes the picture. Furthermore, combined with  \cite{C09}[Theorem 7.3], one can fully characterize the case when $\infty$ is an entrance boundary and $0$ an exit boundary.

\begin{corollary}
Suppose Assumption \ref{a:drift} holds. If $0$ is an exit boundary, then $\infty$ is entrance if and only if there is precisely one quasistationary distribution $\bar \nu$. Furthermore, $\bar\nu$ attracts all initial distributions $\nu$, i.e. for any $\nu$
\[
\lim_{t\to \infty}\Pr^\nu (X_t\in\cdot~|~T_0>t)=\bar\nu(\cdot).
\]
\end{corollary}

Our strategy is similar to the one from \cite{KS12} but in contrast to \cite{KS12}, where the authors could rely on analytic results available in the literature, it is necessary to develop, due to the singularity at $0$, a good deal of results concerning the spectral representation of $L$. The main problem is the fact that the cone of positive $\lambda_0$-harmonic functions might be two-dimensional. We have to choose a proper $\lambda_0$-harmonic function. Our analytic consideration will dictate the choice and finally allow us to characterize the existence of quasistationary distributions. For clarity we chose to prove most of the results for diffusions without internal killing. However, we expect that one can easily generalize our results to also cover killed diffusions (see Remark \ref{r:killing}).

\subsection{Comparison with existing results} We briefly discuss how our results are related to the existing literature.

We we were partly inspired by \cite{C09} - this is one of the few papers which allows $0$ to be a singular boundary. One of the aims of \cite{C09} was to understand the existence and uniqueness of quasistationary distributions for the generalized Feller diffusion $(Z_t)_{t\geq 0}$ that is given by \eqref{e:feller}. We present their results for completeness. Their assumptions are as follows:

\begin{itemize}
  \item [(H0)] The drift satisfies $b\in C^1((0,\infty))$.
  \item [(H1)] For all $x>0$, $\Pr^x(T_0<T_\infty)=1$.
  \item [(H2)] The drift satisfies $\inf_{s>0} (b^2(s)-b'(s))>-\infty$ and $\lim_{s\to\infty} (b^2(s)-b'(s))=+\infty$.
  \item [(H3)] The drift satisfies $$\int_0^1 \frac{1}{b^2(s)-b'(s)+C+2}\rho(s)\,ds<\infty$$ where $C=-\inf_{s>0} (b^2(s)-b'(s))<\infty$.
  \item [(H4)] The drift satisfies $\int_1^\infty \rho(x)\,dx<\infty$ and $\int_0^1 x\sqrt{\rho(x)}\,dx<\infty$.
  \item [(H)] Assumptions (H0), (H1) and (H2) hold and Assumption (H3) or (H4) holds.
\end{itemize}

The authors of \cite{C09} were able to prove the following result about the existence of quasistationary distributions for singular drift functions.

\begin{theorem}[Theorem 5.2 from \cite{C09}]\label{t:C09}
Assume that hypothesis (H) holds. Then there is a quasistationary distribution $\nu_1$ that is the unique Yaglom limit. Furthermore, $\nu_1$ attracts all the compactly supported initial distributions of $X_0$.
\end{theorem}

We are able to generalize Theorem \ref{t:C09} and prove existence results under much weaker assumptions. The most important result of our work consists in the fact that our assumptions do not imply the discreteness of the spectrum of $L$. Note that the assumption
\[
\lim_{x\to \infty} (b(x)^2-b'(x))=\infty
\]
from (H2) ensures the discreteness of the spectrum of $L$ and thus rules out some interesting examples, where the rate of convergence to quasistationarity is not geometric.

In Section \ref{s:assumption} we show the following.

\begin{restatable}{proposition}{assumption}\label{p:assumption}
Suppose that $b\in C^1((0,\infty))$, $\inf_{s>0}(b(s)^2-b'(s))>-\infty$ and
\begin{equation}\label{e:int}
\int_0^1 s\sqrt{\rho(s)}ds<\infty.
\end{equation}
Then Assumption \ref{a:1} holds.
\end{restatable}

Condition \eqref{e:int} is part one part of condition (H4) and the condition $\inf_{s>0}(b(s)^2-b'(s))>-\infty$ is one half of condition (H2). We want to point out that in addition to these assumptions on the behaviour of the drift near $0$ the authors of \cite{C09} also have to assume conditions concerning the behaviour of the drift near $\infty$ (see condition (H2)). We are able to remove the assumptions on the behaviour of the drift near $\infty$.  Thus as far as the existence of quasistationary distributions and the existence of the Yaglom limits is concerned  our conditions are weaker than the ones in \cite{C09}.

Also, contrary to the situation discussed in \cite{C09} the behaviour of
\[
\Pr^x(X_t\in I, T_0>t), I\subset(0,\infty)~\text{an interval}
\]
is in general not exactly exponential as $t\to \infty$.

In the particular case of interest from \cite{C09}, the generalized Feller diffusion, the existence of a Yaglom limit for $(Z_t)_{t\geq 0}$ is equivalent to the existence of a Yaglom limit for solutions of $(Y_t)_{t\geq 0}$ of the transformed stochastic differential equation \eqref{e:Feller2}. As a result, it suffices to understand the simpler operator
\begin{equation}\label{e:L^h}
L^h = -\frac{1}{2}\frac{d^2}{dx^2} + b^h(\cdot)\frac{d}{dx},
\end{equation}
where the drift has the form $b^h(x) = \frac{1}{2x}- \frac{2h(\gamma x^2/4)}{\gamma x}$. This motivated the study of convergence to a quasistationary distribution of general perturbed Bessel processes of the type
\[
\tilde L^h = -\frac{1}{2}\frac{d^2}{dx^2} + \frac{a}{x}\frac{d}{dx} + c(x) \frac{d}{dx}
\]
discussed in Remark 4.6 of \cite{C09}.

Theorem \ref{t:exit_entrance} first appeared in the second author's PhD thesis. The same result was independently shown in \cite{L12}. We include this result as it follows in a rather straightforward way given the general spectral theoretic tools developed in this work and previous ideas from \cite{C09} and \cite{KS12}. This makes this paper self-contained and all known results can be deduced from the ones presented in this paper.

For some recent work on quasistationary distributions with $0$ exit or regular and $\infty$ entrance we refer the reader to \cite{CV15, CV152, CV14}. Their approach is based on probabilistic and coupling methods while ours is based on spectral theory. We note that in \cite{CV152} the authors can study general one-dimensionanl diffusions that are characterized by their speed measure, scale function and killing measure. As such they are able to analyze the existence and uniqueness of quasistationary distributions for diffusions which are not solutions to SDEs. Our spectral methods would not work in their setting. However, our methods can tackle the problem when $\infty$ is a natural boundary while the methods from \cite{CV15, CV152} have not been able to cover this situation. Partial results for the exponential convergence to a quasistationary distribution for multi-dimensional Markov processes killed at the boundary of a Riemannian manifold are obtained in \cite{CCV16}. Other recent developments involving quasistationary distributions of Markov processes can be found in \cite{DM14,BC15, CCM16}

\subsection{Outline} The paper is organized as follows. We divide the arguments into two parts separating the analytic and the probabilistic results. In Section \ref{s:analytic} we are going to discuss the analytic problems. The spectral decomposition for singular one-dimensional diffusions is explored -- for $\infty$ entrance the spectrum is analyzed in more detail. Section \ref{s:Yaglom} is dedicated to the probabilistic results regarding Yaglom limits. The case when $\infty$ is natural is treated in Section \ref{s:natural} while the case when $\infty$ is entrance is in Section \ref{s:entrance}. Applications to generalized Bessel processes appear in \ref{s:bessel}. Finally, in Section \ref{s:bio} we apply our results to an example from population dynamics.

\section{Analytic results}\label{s:analytic}
\subsection{One-dimensional diffusions on the half-line}\label{s:diffusions}In the following we specify our setting. We will use the following notation.
\begin{itemize}
\item $C_c^\infty(\R)$ is the set of smooth functions on $\R$ with compact support.
\item $L^p((0,\infty),\rho):=\left\{f:(0,\infty)\to\R, ~f~\text{measurable and}~~\int_0^\infty|f(x)|^p\,\rho(dx)<\infty\right\}$.
\item If $L$ is a selfadjoint operator then we denote by $\sigma(L)$ its spectrum and by $(E_L(\lambda))_{\lambda\in \R}$ its spectral resolution.
\item $\langle f,g\rangle_{L^2((0,\infty),\rho)}:= \int_0^\infty fg \rho(dx)$ is the inner product on the Hilbert space $L^2((0,\infty),\rho)$.
\item If $q$ is a Dirichlet form we denote its domain by $\cD(q)$.
\end{itemize}

Suppose $b:(0,\infty)\to \R$ is a continuous function. Consider the quadratic form $q$ that is defined as the closure of the map
\[
C_c^\infty(\R)\ni \phi \mapsto \frac{1}{2}\int_0^\infty|\phi'(x)|^2\,\rho(dx),
\]
where $\rho(dx)=\rho(x)\,dx=e^{-2\int_1^x2b(s)\,ds}\,dx$. It is easy to see that $q$ is a Dirichlet form. The generator of this Dirichlet form is given by the unique selfadjoint extension associated to the form $q$ in $L^2((0,\infty),\rho)$. This selfadjoint realization of the formal differential operator
\begin{equation}\label{e:tau}
\tau := -\frac{1}{2}\frac{d^2}{dx^2}+b(x)\frac{d}{dx}
\end{equation}
will be denoted by $L$. We define $e^{-tL}$ using spectral calculus via
\[
e^{-tL}g = \int_{\sigma(L)} e^{-t\lambda} E_L(d\lambda)g, ~g\in L^2((0,\infty),\rho).
\]
By Stone's formula the spectral resolution of $L$ can be calculated from the resolvent $(L-z)^{-1}, z\in \C$ via
\begin{equation*}
\begin{split}
&\left\langle f,E_L((\lambda_1,\lambda_2])g\right\rangle_{L^2((0,\infty),\rho)}\\
&= \lim_{\delta\to 0+}\lim_{\varepsilon\to 0+} \int_{\lambda_1+\delta}^{\lambda_2+\delta} d\lambda \left\langle f, \left[(L-(\lambda+i\varepsilon))^{-1} - (L-(\lambda-i\varepsilon))^{-1}\right]g \right\rangle_{L^2((0,\infty),\rho)}.
\end{split}
\end{equation*}
Let $\lambda_0:=\inf \text{spec} (L)$ denote the bottom of the spectrum of $L$. From the general theory of symmetric Markov semigroups we know that $e^{-tL}$ generates a consistent family of strongly continuous semigroups on $L^p((0,\infty),\rho), 1\leq p<\infty$. Moreover, there exists a continuous transition function
\[
(0,\infty)^3\ni (t,x,y)\mapsto p(t,x,y)\in (0,\infty)
\]
which is symmetric in $x$ and $y$, such that for all $f\in L^p((0,\infty),\rho)$ the semigroup can be represented as
\[
e^{-tL}f(x) = \int_0^\infty p(t,x,y)f(y)\rho(dy).
\]
The general theory of Dirichlet forms (see for example \cite{FOT10}) implies that there is a $\rho$-symmetric Hunt process $(X_t)_{t\geq 0}$ with continuous paths such that for every positive starting point and every $f\in C_c^\infty((0,\infty))$ the process $(M_t)_{t\geq 0}$ defined by
\[
M_t:=f(X_t)-f(X_0)-\int_0^tLf(X_s)\,ds
\]
is a martingale (up to some explosion time) with respect to the canonical filtration of $(X_t)_{t\geq 0}$. One also gets that $(M_t)_{t\geq 0}$ has quadratic variation $\int_0^t|f'(X_s)|^2\,ds$. Moreover, a stochastic representation of the semigroup is given by
\[
e^{-tL} f(x) = \E^x \left[f(X_t), T_\infty\wedge T_0>t \right]
\]
for $f\in C_c^\infty((0,\infty))$ and $T_\infty$ denoting the explosion time. Later on we will assume that $\infty$ is inaccessible, in which case $\Pr^x(T_\infty=\infty)=1$.
\begin{remark}
The reader who is not familiar with the theory of Dirichlet forms should assume that $b$ is locally Lipschitz and should think of the process $(X_t)_{t\geq 0}$ as the solution to the stochastic differential equation

\[
dX_t = dB_t-b(X_t)dt.
\]
\end{remark}
We only consider the case where in the Feller classification of boundary points $0$ is an \textit{exit} point which is characterized by
\begin{equation}\label{e:exit}
\int_1^0\left(\int_1^t\rho(s)ds\right)\rho(t)^{-1}dt<\infty ~\text{and}~~\int_1^0\left(\int_1^t\rho(s)^{-1}ds\right)\rho(t)dt=\infty.
\end{equation}
We mainly look at the case when $\infty$ is \textit{inaccessible}.
Condition \eqref{e:exit} implies that $\int_0^1 \rho(t)\,dt=\infty$ and $\int_0^1 \rho(t)^{-1}\,dt<\infty$.
For any $y\in (0,\infty)$ let
\[
T_y:=\inf\{t\geq 0:X_t=y\}
\]
be the first hitting time of $y$. If $\infty$ is natural one has
\[
\lim_{x\to \infty}\Pr^x(T_y<t)=0, y\in (0,\infty), t>0
\]
whereas if $\infty$ is entrance one has
\[
\lim_{x\to \infty}\Pr^x(T_y<t)>0, y\in (0,\infty), t>0.
\]
Note that in the latter case there is a more hands-on equivalent formulation in terms of the drift function (see Remark 7.4 of \cite{C09}): If $b(x)$ goes to infinity and $b'(x)>0$ for $x>x_0$ for some $x_0>0$, then $\infty$ being an entrance boundary is equivalent to $\int_1^\infty \frac{1}{b(x)}\,dx<\infty.$

\begin{remark}\label{r:limitpoint}
We say the Sturm-Liouville expression $\tau$ is in the limit point case at infinity, if for some $c \in (0,\infty)$ there exists a unique selfadjoint extension of $\tau$ satisfying Dirichlet boundary conditions at $0$. This is equivalent to saying that for all $z \in \mathbb{C}\setminus \mathbb{R}$ the ordinary differential equation
\begin{displaymath}
\tau u = z u
\end{displaymath}
has (up to constant multiples) a unique solution $u$, which satisfies
\begin{displaymath}
\int_c^{\infty}|u(x)|2\rho(dx) < \infty.
\end{displaymath}
An analogous assertion applies to limit point case at $0$. A detailed account of this topic can be found in Section 13.3 of \cite{WII}.
If $0$ is exit and $\infty$ is inaccessible the differential expression $\tau$ from \eqref{e:tau} is in the limit point case at $0$ and at $\infty$ (see Definition 2.4 in \cite{KS12}), i.e. the restriction of $L$ to $C_c^\infty(0,\infty)$ is even essentially selfadjoint and hence has a unique selfadjoint extension. The fact that we are in the limit point case at $\infty$ can be found in Lemma 3.1 from \cite{KS12} while the analogous statement for $0$ is proved in \cite{W85l}.
\end{remark}

In what follows we develop the spectral theory for the differential expression
\[
\tau = -\frac{1}{2}\frac{d^2}{dx^2} + b(x)\frac{d}{dx}
\]
with $b$ having a singularity at $0$. If there is no singularity at $0$ this has already been done in \cite{M61}, \cite{ES07} and \cite{KS12}. In \cite{GZ06} the spectral theory was carried out for Schr\"{o}dinger operators with strongly singular potential $V$, i.e.
\[
-\frac{1}{2}\frac{d^2}{dx^2} +V(x)
\]
where $V$ is not locally integrable at $0$. The straightforward approach to reduce the case of drift to the case of a potential by exploiting the unitary equivalence of $\tau$ to
\[
-\frac{1}{2}\frac{d^2}{dx^2} +[b^2-b'](x)
\]
was carried out in \cite{C09}. The disadvantage was that additional restrictions on $b^2-b'$ were required. Namely, one had to assume that
\[
C= -\inf_{y\in(0,\infty)}(b(y)^2-b'(y))<\infty ~\text{and}~\lim_{y\to \infty} (b(y)^2-b'(y))=\infty.
\]
We circumvent these additional assumptions on the behaviour of $b$ close to $0$ and $\infty$ by doing the spectral theory directly. In order to be as selfcontained as possible we provide all the details even though some of them are known in the analytic literature.

We start by proving a spectral theorem for diffusions with exit boundary at $0$ (Theorem \ref{t:exit_spectral}) and then show that if, in addition, $\infty$ is an entrance boundary then the spectrum is discrete.

\subsection{Spectral Theorem for diffusion operators with singularity at $0$} Our aim is to show that one can transfer the results of Section 3 of \cite{GZ06} to diffusions with singular drift (instead of a singular potential). The idea can be essentially summarized in the following scheme:
\begin{itemize}
\item Using methods from functional analysis and the theory of ordinary differential equations one deduces a formula for the resolvent of the selfadjoint realization of $\tau$.
\item Recall from \cite{WI} (e.g. Satz B.2 in \cite{WI}) that the resolvent $R(z)$ is a Herglotz function function, i.e. it is analytic on the upper half plane, satisfies $\Im R(z) > 0$ (where $\Im z$ denotes the imaginary part of the complex number $z$)  and $|\Im(z) \cdot R(z)| \leq M$ for some $M<\infty$ and all $z$ with $\Im z>0$. Using a theorem of Herglotz (see Satz B.2 in \cite{WI}) one concludes that the resolvent can be represented as a Stieltjes tranform of a measure.
\end{itemize}
If the function $f$ is a scalar valued Herglotz function then the theorem of Herglotz tells us that there is a right continuous non-decreasing function $w$ such that
\begin{displaymath}
f(z):=\int_{\mathbb{R}}\frac{1}{t-z}\,dw(t)
\end{displaymath}
and
\begin{displaymath}
w(t) = \lim_{\delta\rightarrow 0+}\lim_{\varepsilon \rightarrow 0+}\frac{1}{\pi}\int_{-\infty}^{t+\delta} \Im f(s+i\varepsilon)\,dt
\end{displaymath}
Using this result and properties of the resolvent we will be able to establish a spectral representation of the selfadjoint realization of $\tau$. With this in mind we define the concept of a Weyl-Titchmarsh solution.
\begin{definition}\label{d:vp}
A solution $\vp$ of
\begin{equation}\label{e:tau_sol}
(\tau\psi)(z,x)=z\psi(z,x)
\end{equation}
is called an analytic Weyl-Titchmarsh solution in $B_R:=\{z\in\C: |z|<R\}$, if $\vp$ satisfies the following conditions
\begin{itemize}
\item[i)] For every fixed $z\in B_R$, $\vp(z,\cdot)$ is a non-trivial solution to \eqref{e:tau_sol} and for every fixed $x\in (0,\infty)$ the function $\vp(\cdot,x)$ is analytic in $B_R$.
\item[ii)] If $(z,x)\in (-R,R)\times (0,\infty)$ then $ \vp(z,x) \in \R$.
\item[iii)]  For every $a\in (0,\infty)$ and $z\in B_R$ one has
\[
\int_0^a |\vp(z,x)|^2\rho(dx)<\infty.
\]
\end{itemize}
\end{definition}

The system of fundamental solutions of \eqref{e:tau_sol} is given by functions $\varphi(z,\cdot,x_0)$ and $\theta(z,\cdot,x_0)$ solving
\begin{equation*}
\begin{split}
(\tau\varphi)(z,x) &= z\varphi(z,x),\\
\varphi(z,x_0,x_0)&=\rho(x_0)\theta'(z,x_0,x_0)=0,\\
\rho(x_0)\varphi'(z,x_0,x_0)&=\theta(z,x_0,x_0)=1,
\end{split}
\end{equation*}
for $z\in\C$ and a fixed reference point $x_0\in(0,\infty)$. For every fixed $x\in (0,\infty)$ the solutions $\varphi(z,x,x_0)$ and $\theta(z,x,x_0)$ are analytic with respect to $z\in \C$ and satisfy
\[
W(\theta(z,\cdot,x_0), \varphi(z,\cdot,x_0))(x)=1,
\]
where $W(f,g)= f(x)\rho(x)g'(x)-\rho(x)f'(x)g(x)$ denotes the Wronskian of $f$ and $g$.

Later we will use the fact that for any fixed $x\in(0,\infty)$ the function $\vp(z,x)$ is analytic in $z\in B_R$. For this observe that
\[
\vp(z,x) = \bar \varphi(z,x):= \vp(z,x_0)\rho(x_0)\varphi(z,x,x_0)+\vp(z,x_0)\theta(z,x,x_0),
\]
as both sides of the above equation are solutions of $\tau u =zu$ satisfying $\vp(z,x_0)=\bar\varphi(z,x_0)$ and $\rho(x_0)\vp(z,x_0)= \rho(x_0)\bar\varphi(z,x_0)$. Differentiating both sides of the last equation shows the required analyticity of $\vp(z,x)$ in $z\in B_R$.

We next introduce the Weyl-Titchmarsh solutions $\psi_{\pm}(z,\cdot,x_0)$, $x_0\in(0,\infty), z\in \C\setminus\R$, of \eqref{e:tau_sol}. As argued in Remark \ref{r:limitpoint} we are in the limit-point case at $0$ and at $\infty$. This implies that the Weyl-Titchmarsh solutions are up to constant multiples characterized by
\begin{equation*}
\begin{split}
\psi_-(z,\cdot,x_0)\in L^2((0,x_0),\rho),\\
\psi_+(z,\cdot,x_0)\in L^2((x_0,\infty),\rho)
\end{split}
\end{equation*}
for $z\in\C\setminus\R$. We normalize $\psi_{\pm}(z,\cdot,x_0)$ by requiring
\[
\psi_{\pm}(z,x_0,x_0)=1.
\]
This yields
\[
\psi_{\pm}(z,x,x_0) = \theta(z,x,x_0) + m_\pm(z,x_0)\varphi(z,x,x_0), ~x,x_0\in (0,\infty), z\in\C\setminus\R
\]
where $m_\pm(z,x_0)$ is defined by
\[
m_\pm(z,x) = \frac{\rho(x)\psi_\pm'(z,x,x_0)}{\psi_\pm(z,x,x_0)}.
\]
It is well-known that $m_\pm(z,x)$ are Herglotz- and anti-Herglotz-functions (see \cite{M61, WI} for more properties of Herglotz functions), respectively.

\begin{lemma}\label{l:1}
Assume that there exists a Weyl-Titchmarsh solution $\vp(z,x)$ to \eqref{e:tau_sol} that is analytic in $B_R$. Then there exists a solution $\tilde\theta(z,x)$ of \eqref{e:tau_sol} having the following properties:
\begin{enumerate}
\item For every $x\in(0,\infty)$ the function $\tilde\theta(\cdot,x)$ is analytic in $z\in B_R$.
\item For every $x\in(0,\infty), z\in \R$ one has $\tilde\theta(z,x)\in \R$.
\item For every $x\in(0,\infty), z\in B_R$
\[
W(\tilde\theta(z,\cdot),\vp(z,\cdot))\equiv 1.
\]
\end{enumerate}
\end{lemma}
\begin{proof}
We follow the proof of Lemma 3.3 from \cite{M61}. Fix $x_0\in (0,\infty)$ and consider
\[
\tilde\theta(z,x) = \frac{\rho(x_0)\vp'(z,x_0)}{\vp(z,x_0)^2+(\rho(x_0)\vp'(z,x_0))^2}\theta(z,x,x_0) - \frac{\vp(z,x_0)}{\vp(z,x_0)^2+(\rho(x_0)\vp'(z,x_0))^2}\varphi(z,x,x_0)
\]
Since the solution $\vp(z,\cdot)$ is non-trivial for every $z\in B_R$ we have $\vp(z,x_0)\neq 0$ and $\rho(x_0)\vp'(z,x_0)\neq 0$ for every $z\in B_R$. Thus we conclude that $\tilde \theta(z,\cdot)$ is well defined for $z\in B_R$. Due to the properties of $\vp(z,x_0), \theta(z,x,x_0)$ and $\varphi(z,x,x_0)$ the function $\tilde \theta(z,\cdot)$ is analytic in $z\in B_R$ for every fixed $x\in (0,\infty)$. Moreover, we have
\begin{equation*}
\begin{split}
W(\tilde\theta(z,\cdot),\vp(z,\cdot))(x) &= W(\tilde\theta(z,\cdot),\vp(z,\cdot))(x_0)\\
&= \tilde\theta(z,x_0)\rho(x_0)\vp'(z,x_0)-\rho(x_0)\tilde\theta'(z,x_0)\vp(z,x_0)\\
&= \frac{(\rho(x_0)\vp'(z,x_0))^2}{\vp(z,x_0)^2+(\rho(x_0)\vp'(z,x_0))^2} + \frac{(\vp(z,x_0))^2}{\vp(z,x_0)^2+(\rho(x_0)\vp'(z,x_0))^2}\\
&=1.
\end{split}
\end{equation*}
\end{proof}
Let $\tilde m_+(z)$ be such that
\[
\tilde\psi_+(z,x) = \tth(z,x) + \tilde m_+(z)\vp(z,x), ~x\in(0,\infty)
\]
is a solution to \eqref{e:tau_sol} that satisfies for all $a\in(0,\infty), z\in B_R\setminus\R$
\[
\tilde\psi_+(z,\cdot) \in L^2((a,\infty),\rho).
\]
Since the differential expression is in the limit point case at infinity the solution $\tilde\psi_+(z,\cdot)$ is proportional to $\psi_+(z,\cdot,x_0)$. This gives
\[
m_+(z,x) = \frac{\rho(x)\tth'(z,x)+\tilde m_+(z)\rho(x)\vp'(z,x)}{\tth(z,x)+\tilde m_+(z)\vp(z,x)}
\]
and by direct computation
\begin{equation}\label{e:tilde_m}
\begin{split}
\tilde m_+(z) &= \frac{m_+(z,x)\tilde\theta(z,x)-\rho(x)\tth'(z,x)}{\rho(x)\vp(z,x)-m_+(z,x)\vp(z,x)} = \frac{W(\vp(z,\cdot),\psi_+(z,\cdot,x_0))}{W(\vp(z,\cdot),\psi_+(z,\cdot,x_0))}\\
&= \frac{\tth(z,x)}{\vp(z,x)}\frac{m_+(z,x)}{m_-(z,x)-m_+(z,x)} -\frac{\rho(x)\tth'(z,x)}{\vp(z,x)}\frac{1}{m_-(z,x)-m_+(z,x)}
\end{split}
\end{equation}
Since $m_-$ and $m_+$ are Herglotz- and anti-Herglotz functions respectively, $\vp(x,z)$ and $\tth(z,x)$ are analytic, and since $\vp(z,x\neq 0$ we see that $\tilde m_+(z)$ is analytic in $z\in B_R$. Note that $\vp(z,x)$ is not $0$, since if there would be an $x_0$ with $\vp(z,x_0)=0$, then the function $\vp(z,\cdot)$ would be an eigenfunction to a non-real eigenvalue of the selfadjoint realization of $\tau$ in $L^2((0,x_0),\rho)$ with Dirichlet boundary condition at $x_0$.

A direct computation for $z\in B_R\setminus \sigma(L)$ and $x,y\in (0,\infty)$ shows that the Green's function $G(z,x,y)$ is given by
\begin{equation}\label{e:Green}
G(z,x,y)= \left\{
	\begin{array}{ll}
		\vp(z,x)\tilde\psi_+(z,y)  & \mbox{if } 0<x\leq y, \\
		\vp(z,y)\tilde\psi_+(z,x) & \mbox{if } 0<y\leq x.
	\end{array}
\right.
\end{equation}
This means that for every $x\in(0,\infty)$ and $f\in L^2((0,\infty),\rho)$
\begin{equation}\label{e:res}
((L-z)^{-1}f)(x) = \int_0^\infty G(z,x,y)f(y)\rho(dy).
\end{equation}

The next lemma collects some important properties of $\tilde m_+$. These will be crucial for the spectral decomposition that is carried out in the next theorem.

\begin{lemma}\label{l:2}
Assume that there exists an analytic Weyl-Titchmarsh solution of \ref{e:tau_sol} in $B_R$. Then the function $\tilde m_+$ satisfies the following conditions:
\begin{enumerate}
\item [i)] $\tilde m_+(z) = \overline{\tilde m_+(\bar z)}$ for $z\in B_R$,
\item [ii)] For any $\varepsilon_0>0$ and $\lambda_1<\lambda_2$ there exists a constant $C(\lambda_1,\lambda_2,\epsilon_0)$ such that $\varepsilon|\tilde m_+(\lambda+i\varepsilon)|\leq C(\lambda_1,\lambda_2,\epsilon_0)$ for all $\lambda\in[\lambda_1, \lambda_2], 0<\varepsilon\leq \varepsilon_0$,
\item [iii)] For any $\varepsilon_0>0$ and $\lambda_1<\lambda_2$ one has $\varepsilon|\text{Re}\left(\varepsilon(\tilde m_+(\lambda+i\varepsilon)\right)| = o(1)$ for all $\lambda\in[\lambda_1, \lambda_2], 0<\varepsilon\leq \varepsilon_0$,
\item [iv)] The limit
\[
\lim_{\varepsilon\downarrow 0} (i\ve)\tilde m_+(\lambda + i\ve) = \lim_{\varepsilon\downarrow 0}\ve \Im(\tilde m_+(\lambda + i\ve))
\]
exists for all $\lambda\in (-R,R)$ and is nonnegative,
\item [v)] For a.e. $\lambda\in [\lambda_1,\lambda_2]$ one has that
\[
\tilde m_+(\lambda+i0) = \lim_{\varepsilon\downarrow 0} \tilde m_+(\lambda + i\ve)
\]
exists and $\Im (\tilde m_+(\lambda+i0))\geq 0$.
\end{enumerate}
Moreover, there exists a measure $\sigma$ such that
\begin{equation}\label{e:sigma}
\sigma((\lambda_1,\lambda_2]) = \lim_{\delta\downarrow 0}\lim_{\ve\downarrow 0}\frac{1}{\pi} \int_{\lambda_1+\delta}^{\lambda_2+\delta} \Im (\tilde m_+(\lambda+i\ve))d\lambda.
\end{equation}
\end{lemma}
\begin{proof}
Since for $(\lambda,x)\in (-R,R)\times (0,\infty)$ the numbers $\vp(\lambda,x)$ and $\tth(\lambda,x)$ are real, we get for $(z,x)\in B_R\times (0,\infty)$
\[
\vp (z,x)= \overline{\vp(\bar z,x)}~~,~~ \tth(z,x)= \overline{\tth(\bar z,x)}.
\]

Let $c,d\in (0,\infty)$ with $c<d$. Then
\begin{equation}\label{e:1_c<d}
\begin{split}
\int_{\sigma(L)} \frac{d\|E_L(\lambda)\ind_{[c,d]}\|^2_{L^2((0,\infty),\rho)}}{\lambda-z} &= \left\langle\ind_{[c,d]}, (L-z)^{-1}\ind_{[c,d]}\right\rangle_{L^2((0,\infty),\rho)}\\
&=\int_c^d \rho(dx)\int_c^x\rho(dy)\tth(z,x)\vp(z,y)\\
&~~+ \int_c^d \rho(dx)\int_x^d\rho(dy)\tth(z,y)\vp(z,x)\\
&~~+\tilde m_+(z) \left[\int_c^d\rho(dx)\vp(z,x)\right]^2.
\end{split}
\end{equation}
Observe that we can find $c,d$ in such a way that
\[
\int_c^d \vp(z,x)\rho(dx)\neq 0
\]
for all $z$ in a sufficiently small complex open neighborhood of $[\lambda_1,\lambda_2]$. Note that equation \eqref{e:1_c<d} implies i). Define the function $H$ via
\[
\C\setminus\sigma(L)\ni z \mapsto \int_{\sigma(L)}  \frac{d\|E_L(\lambda)\ind_{[c,d]}\|^2_{L^2((0,\infty),\rho)}}{\lambda-z}.
\]
Then $H$ is a Herglotz function and ii)-iv) follow from the basic properties of $\tth(z,x)$ and $\vp(z,x)$. By the Herglotz property of $H$ we also have
\[
 \lim_{\delta\downarrow 0}\lim_{\ve\downarrow 0}\frac{1}{\pi} \int_{\lambda_1+\delta}^{\lambda_2+\delta} \Im \left( \int_{\sigma(L)} \frac{d\|E_L(\lambda)\ind_{[c,d]}\|^2_{L^2((0,\infty),\rho)}}{\lambda-(x+i\ve)}\right)\,dx = \|E_L((\lambda_1,\lambda_2])\ind_{[c,d]}\|^2_{L^2((0,\infty),\rho)}
\]

We next apply \eqref{e:1_c<d} to $\lambda+i\ve$ with $\lambda\in (\lambda_1,\lambda_2)$ and $\ve >0$ small enough. Then the quantity $\sigma$ from \eqref{e:sigma} satisfies
\begin{equation*}
\begin{split}
\sigma((\lambda_1,\lambda_2]) &= \lim_{\delta\downarrow 0}\lim_{\ve\downarrow 0}\frac{1}{\pi} \int_{\lambda_1+\delta}^{\lambda_2+\delta} \Im (\tilde m_+(\lambda+i\ve))d\lambda\\
&= \lim_{\delta\downarrow 0}\lim_{\ve\downarrow 0}\frac{1}{\pi} \int_{\lambda_1+\delta}^{\lambda_2+\delta}d\lambda \Im  \Bigg\{ \int_{\sigma(L)}\frac{d\|E_L(\tilde \lambda)\ind_{[c,d]}\|^2_{L^2((0,\infty),\rho)}}{\tilde \lambda-(\lambda+i\ve)}\\
&~~~\times \left[\left(\int_{c_0}^{d_0}\rho(dx)\vp(\lambda,x)\right)^2+2i\ve\left(\int_{c_0}^{d_0}\rho(dx)(d/dz)\vp(z,x)\Big|_{z=\lambda}\right)+O(\ve^2)\right]^{-1}+O(\ve)\Bigg\}\\
&= \int_{(\lambda_1,\lambda_2]} d\|E_L(\lambda)\ind_{[c,d]}\|^2_{L^2((0,\infty),\rho)}\left[\int_{c_0}^{d_0}\rho(dx)\vp(\lambda,x)\right]^{-2}
\end{split}
\end{equation*}
and therefore defines a measure on $\R$.
\end{proof}
Part iv) of the following theorem connects the measure $\sigma$ from Lemma \ref{l:2} with the spectral measure of the operator $L$.
\begin{theorem}\label{t:exit_spectral}
Assume that $0$ is an exit boundary. Then the following hold:
\begin{enumerate}
\item [i)] For every $\lambda\in \R$ the ordinary differential equation $(\tau-\lambda)u=0$ has a fundamental system $(u_\lambda^1,u_\lambda^2)$ of solutions with
\[
\lim_{x\to 0} u_\lambda^1(x)=0 ~\text{and}~\lim_{x\to 0} u_\lambda^2(x)=0.
\]
Furthermore, $u_\lambda^1\in L^1((0,\infty),\rho)\cap L^2((0,\infty),\rho)$.
\item [ii)] For every fixed $R>0$ there exists a Weyl-Titchmarsh solution $\vp(z,x)$ for $L$, that is analytic in $B_R$.
\item [iii)] Let $R>0$ and $\vp(\lambda,\cdot)$ be as in ii). If $\lambda_0 =\inf \text{spec}(L)$, then $\vp(\lambda_0,\cdot)$ is non-negative.
\item [iv)] Let $R>0$ and $\vp(\lambda,\cdot)$ be as in ii). There is a measure $\sigma=\sigma_R$ on $\R$, such that for every $F\in C(R)$, every $f,g\in C_c^\infty((0,\infty))$, and every $-R<\lambda_1<\lambda_2<R$
\[
\langle f,F(L)E_L((\lambda_1,\lambda)2])g\rangle_{L^2((0,\infty),\rho)} = \left\langle\hat f, M_{F\ind_{(\lambda_1,\lambda_2]}}\hat g\right\rangle_{L^2(\R,\sigma)},
\]
where
\[
\hat h(\lambda):= \int_0^\infty \vp(\lambda,x)h(x)\rho(dx)
\]
for $h\in C_c^\infty((0,\infty))$ and $M_{F\ind_{(\lambda_1,\lambda_2]}}$ denotes the bounded operator, which acts by multiplication with the function $F\ind_{(\lambda_1,\lambda_2)}$.
\item [v)] Let $R>0$ and $\vp(\lambda,\cdot)$ be as in ii). Then
\[
E([\lambda_0,\lambda_1])e^{-tL}\ind_A(x)=\int_A h^{\lambda_1}(t,x,y)\,dy,
\]
where
\[
 h^{\lambda_1}(t,x,y):= \int_{\lambda_0}^{\lambda_1}e^{-t\lambda}\vp(\lambda,x)\vp(\lambda,y)\sigma(d\lambda).
\]
\end{enumerate}
\end{theorem}
\begin{proof}
We first construct the fundamental system for part i). We have to show that there exists an anayltic Weyl-Titchmarsh function $\vp(z,\cdot), z\in B_R$. For $z\in B_R$ we construct a solution $u^R(z,x)$ of the equation $(L-z)u=0$ in a certain neighborhood $U$ of $0$, which for every fixed $x\in U$ is analytic in $z\in B_R$ and which for every fixed $z\in B_R$ satisfies $\lim_{x\to 0}u^R(z,x)=1$. Since the ordinary differential equation $(L-z)u=0$ is linear each $u^R(z,\cdot), z\in B_R$ gives rise to a solution of $(L-z)u=0$ not only on $U$ but also on $(0,\infty)$. As a result it is enough to construct $u^R(z,\cdot)$ only on $U$. This will be done by a well-known iteration procedure.

For $0<\ve<1$ choose $\delta=\delta_R>0$ such that for every $z\in B_R$ and $x\in(0,\delta)$
\[
\int_0^x\rho(r)^{-1}\,dr\int_r^\delta\rho(s)\,ds<\ve.
\]
Set $u_0^R(z,x)=1$ and recursively define
\[
u_{n+1}^R(z,x)=1+z\int_0^x\rho(r)^{-1}\,dr\int_r^\delta u_n^R(z,s)\rho(s)\,ds.
\]
One can then easily see by induction that the sequence $(u_n^R(z,x))_{n\in\N}$ converges uniformly for $x\in[0,\delta_R], z\in B_R$. This implies that the limit
\[
u^R(z,x):=\lim_{n\to \infty}u_n^R(z,x)
\]
is continuous for every $x\in[0,\delta_R]$ and for every fixed $x$ is analytic in $z\in B_R$. Moreover the limit satisfies
\[
u^R(z,x)=1+2z\int_0^x\rho(r)^{-1}\,dr\int_r^\delta u^R(z,s)\rho(s)\,ds
\]
and solves for every $z\in B_R$ the equation $(L-z)u=0$ in $(0,\delta_R)$ with $\lim_{x\to 0}u^R(z,x)=1$. One also has that for $\lambda\in\R\cap B_R$ the quantity $u^R(\lambda,x)$ is real.

Choose $\tilde\delta_R$ such that $u^R(z,x)\neq 0$ for every $(z,x)\in B_R\times (0,\tilde\delta_R)$ and set for every $x\in(0,\tilde\delta_R)$
\[
\vp^R(z,x):=u^R(z,x)\int_0^x u^R(z,y)^{-2}\rho(y)^{-1}dy.
\]
Observe that $\vp^R(\lambda,x)\in\R$ for real $\lambda$ and that $\lim_{x\to 0}\vp^R(z,x)=0$. For any simple closed path $\gamma$ in $B_R$ due to the analyticity of $u^R(z,x)$ in $z\in B_R$
\[
\int_\gamma\left(\int_0^x u^R(z,s)^{-2}\rho(s)^{-1}\,ds\right)dz = \int_0^x\left(\int_\gamma u^R(z,s)^{-2}\rho(s)^{-1}\,dz\right)ds = 0
\]
so by Morera's theorem we see that $\vp^R(z,x)$ is analytic in $z\in B_R$. The two solutions $u^R(z,\cdot)$ and $\vp^R(z,\cdot)$ are linearly independent and the solution $\vp(z,\cdot)$ is also integrable since
\[
\left|\int_0^x\vp^R(z,t)\,dt\right|\leq \frac{\max_{y\in [0,x]}|u^R(z,y)|}{(\min_{y\in [0,x]}|u^R(z,y)|)^2} \int_0^x\rho(t)\int_0^t\rho(s)^{-1}<\infty
\]
and $0$ is an exit boundary. It remains to observat that $\vp^R(z,\cdot)\in L^\infty((0,1),\rho)$. Since bounded and integrable functions are square integrable we get that for every $z\in B_R$ and $a\in (0,\infty)$
\[
\int_0^a |\vp^R(z,x)|^2\rho(dx)<\infty.
\]

For parti iv) we modify the techniques from \cite{M61}. We make use of Stone's formula in order to connect $\sigma$ with the spectral resolution $E_L$ of $L$. Stone's formula implies that
\begin{equation*}
\begin{split}
\langle f, F(L)E_L((\lambda_1,\lambda_2])g,\rangle_{L^2((0,\infty),\rho)} &=\lim_{\delta\downarrow 0}\lim_{\ve\downarrow 0}\frac{1}{2\pi i}\int_{\lambda_1+\delta}^{\lambda_2+\delta}d\lambda F(\lambda)\Big[\langle f,(L-(\lambda+i\ve))^{-1}g\rangle_{L^2((0,\infty),\rho)}\\
&~~~~~~~-\langle f, (L-(\lambda-i\ve))^{-1}g\rangle_{L^2((0,\infty),\rho)}\Big].
\end{split}
\end{equation*}
Using formula \eqref{e:res} for the kernel of the resolvent we get
\begin{equation*}
\begin{split}
&\langle f, F(L)E_L((\lambda_1,\lambda_2])g)_{L^2((0,\infty),\rho)}\\
&= \lim_{\delta\downarrow 0}\lim_{\ve\downarrow 0}\frac{1}{2\pi i}\int_{\lambda_1+\delta}^{\lambda_2+\delta}d\lambda F(\lambda)\Bigg\{\int_0^\infty\rho(dx)\Bigg[ \overline{f(x)}\tilde\psi_+(\lambda+i\ve,x)\int_0^x\vp(\lambda+i\ve,y)g(y)\rho(dy)\\
&~~~+\overline{f(x)}\vp(\lambda+i\ve,x)\int_x^\infty \rho(dy)\tilde\psi_+(\lambda+i\ve,y)g(y)\Bigg] - \Bigg[\overline{f(x)}\tilde\psi_+(\lambda-i\ve,x)\int_0^x\rho(dy)\vp(\lambda-i\ve,y)g(y)\\
&~~~+\overline{f(x)}\vp(\lambda-i\ve,x)\int_x^\infty\rho(dy)\tilde\psi_+(\lambda-i\ve,y)g(y)\Bigg]\Bigg\}.
\end{split}
\end{equation*}
Since all integrals are on bounded sets and since the integrands are continuous we can interchange integrals and limits to get

\begin{equation}\label{e:inner1}
\begin{split}
&\langle f, F(L)E_L((\lambda_1,\lambda_2])g)_{L^2((0,\infty),\rho)}\\
&=\int_0^\infty \rho(dx)\overline{f(x)}\Bigg\{\int_0^x\rho(dy)g(y) \lim_{\delta\downarrow 0}\lim_{\ve\downarrow 0}\frac{1}{2\pi i}\int_{\lambda_1+\delta}^{\lambda_2+\delta}d\lambda F(\lambda)[\tilde\psi_+(\lambda+i\ve,x)\vp(\lambda+i\ve,y)\\
&~~~~~-\tilde\psi_+(\lambda-i\ve,x)\vp(\lambda-i\ve,y)]\\
&~~~~~+\int_0^\infty \rho(dy)g(y)\lim_{\delta\downarrow 0}\lim_{\ve\downarrow 0}\frac{1}{2\pi i}\int_{\lambda_1+\delta}^{\lambda_2+\delta}d\lambda F(\lambda)[\vp(\lambda+i\ve,x)\tilde\psi_+(\lambda+i\ve,y)\\
&~~~~~- \vp(\lambda-i\ve,x)\tilde \psi_+(\lambda-i\ve,y)]\Bigg\}.
\end{split}
\end{equation}
Due to the regularity properties of $\vp$ and $\tth$ we have that as $\ve\downarrow 0$
\begin{equation}\label{e:tthvp}
\begin{split}
\tth(\lambda\pm i\ve) &= \tth(\lambda,x)\pm i\ve \frac{d}{dz}\tth(z,x)\Big|_{z=\lambda} + O(\ve^2),\\
\vp(\lambda\pm i\ve) &= \vp(\lambda,x)\pm i\ve \frac{d}{dz}\vp(z,x)\Big|_{z=\lambda} + O(\ve^2)
\end{split}
\end{equation}
where $O(\ve)$ is locally uniform in $(\lambda,x)$. Therefore we get by \eqref{e:tthvp} and by ii) of Lemma \ref{l:2}
\begin{equation*}
\begin{split}
&\vp(\lambda+i\ve,x)\tilde\psi_+(\lambda+i\ve,y)-\vp(\lambda-i\ve)\tilde\psi_+(\lambda-i\ve,y)\\
&= \vp(\lambda,x)[\tth(\lambda,y)+\tilde m_+(\lambda-i\ve)\vp(\lambda,y)] + \vp(\lambda,x)[\tth(\lambda,y)+\tilde m_+(\lambda-i\ve)\vp(\lambda,y)]+o(1),
\end{split}
\end{equation*}
where $o(1)$ is locally uniform. The previous equation together with \eqref{e:inner1} gives
\begin{equation*}
\begin{split}
&\langle f, F(L)E_L((\lambda_1,\lambda_2])g)_{L^2((0,\infty),\rho)}\\
&~~=\int_0^\infty \rho(dx)\overline{f(x)}\int_0^\infty \rho(dy)g(y)\lim_{\delta\downarrow 0}\lim_{\ve\downarrow 0}\frac{1}{\pi }\int_{\lambda_1+\delta}^{\lambda_2+\delta}d\lambda F(\lambda)\vp(\lambda,y)\Im (\tilde m_+(\lambda+i\ve)).
\end{split}
\end{equation*}
By the definition of the measure $\sigma$ we have for every $h\in C(\R)$
\[
\int_{(\lambda_1,\lambda_2]} h(\lambda)\sigma(d\lambda) = \lim_{\delta\downarrow 0}\lim_{\ve\downarrow 0}\frac{1}{\pi }\int_{\lambda_1+\delta}^{\lambda_2+\delta}d\lambda \Im (\tilde m_+(\lambda+i\ve))h(\lambda)
\]
which yields
\begin{equation*}
\begin{split}
&\langle f, F(L)E_L((\lambda_1,\lambda_2])g)_{L^2((0,\infty),\rho)}\\
&~~~~=\int_0^\infty \rho(dx)\overline{f(x)}\int_0^\infty \rho(dy)g(y)\int_{(\lambda_1,\lambda_2]} \sigma(d\lambda)F(\lambda)\vp(\lambda,x)\vp(\lambda,y).
\end{split}
\end{equation*}
Part v) follows directly from iv) by taking a mollifier at $x$ for $f$ and a smoothing procedure for $g=\ind_A$.

Finally, part iii) can be shown similarly to when there is a regular boundary at $0$ (see \cite{KS12}).
If $\lambda_0$ is an isolated eigenvalue then it follows by Section 17.4 from
\cite{WII} that the function
$\tilde{\varphi}(\lambda_0,\cdot)$ does not change sign and can therefore be
chosen to be non-negative.

Let us assume that there exists $x_0 \in (0,\infty)$ where $\tilde{\varphi}(\lambda_0,\cdot)$ changes sign. Denote
by $L^{x_0}$ the selfadjoint realization of $L$ in $(x_0,\infty)$ with a Dirichlet
boundary at $x_0$ and by $L_{x_0}$ the the selfadoint realization of $L$ in $(0,x_0)$
with a Dirichlet boundary at $x_0$. Using the spectral invariance of the essential
spectrum the direct sum $L_{x_0} \bigoplus L^{x_0}$ has $\lambda_0$ as the
bottom of the spectrum. Using Lemma 2.2 from \cite{ES07} we conclude that
$\tilde{\varphi}(\lambda_0,\cdot)$ does not change sign on $(x_0,\infty)$.

On the other hand using \cite{L12} one notes that the operator $L_{x_0}$ has compact
resolvent and therefore a purely discrete spectrum. Thus
$\tilde{\varphi}(\lambda_0,\cdot)$ is an eigenfunction of  $L_{x_0}$ with eigenvalue
$\lambda_0$ and this must necessarily be the leading eigenvalue. If this were true then for $r>0$ the operator $L^{x_0+r}$ would have the same leading
eigenvalue as $L^{x_0}$ implying in particular that the exit time of the considered diffusion
from $(0,x_0)$ has the same exponential tails as the exit time of the same diffusion
from $(0,x_0+r)$. This is a contradiction.

\end{proof}
\begin{remark}
Observe that for different $0<R_1<R_2$ the solutions $\vp^{R_1}(\lambda,\cdot)$ and $\vp^{R_2}(\lambda,\cdot)$ with $\lambda\in(0,R_1)$ are linearly dependent and thus can only differ by a constant factor $C_\lambda$
\[
\vp^{R_2}(\lambda,\cdot) = C_\lambda \vp^{R_1}(\lambda,\cdot).
\]
This implies that quotients of expressions involving $\vp^R$ are well defined and independent of $R$. In particular measures of the form
\[
\mu_{(0,z)}^R (B) = \frac{\int_B\vp^R(\lambda_0,x)\rho(dx)}{\int_0^z\vp^R(\lambda_0,x)\rho(dx)}
\]
with $z\in(0,\infty]$ and $B$ a Borel subset of $(0,\infty)$  are independent of $R$. In such a situation we omit the dependence of $\vp^R$ on $R$.
\end{remark}

\subsection{Spectral Theorem for entrance boundary at $\infty$} We now consider the case when $\infty$ is an entrance boundary in more detail. It turns out that this simplifies things. In \cite{C09} it was shown that when $\infty$ is an entrance boundary and $0$ is not singular the spectrum of $L$ is discrete. We show how this result can be lifted to a singularity at $0$ by splitting $b$ into the singular part close to $0$ and a regular part away from $0$.

The following theorem gives us in particular an improved version of Theorem \ref{t:exit_spectral} v). It tells us that the probability $\Pr^x(X_t\in A, T_0>t)=e^{-tL}\ind_A(x)$ can be represented globally by the eigenfunctions of $L$.

\begin{theorem}\label{t:exit_entrance_spectrum}
Assume that $0$ is an exit boundary and $\infty$ is entrance. Then the spectrum of $L$ is purely discrete, consisting of countably many eigenvalues $\lambda_0<\lambda_1<\dots,$ and the eigenfunctions $(u_{\lambda_k})$ belong to $L^1((0,\infty),\rho)$. In particular, we obtain the representation
\[
e^{-tL}\ind_A(x)=\int_A p(t,x,y)\,dy
\]
with
\[
p(t,x,y)=\sum_{k\geq 0} e^{-t\lambda_k}u_{\lambda_k}(x)u_{\lambda_k}(y),
\]
where the series converges in  $L^2((0,\infty),\rho)$. Moreover, one has
\[
\lim_{t\to \infty}e^{\lambda_0 t}p(t,x,y)=u_{\lambda_0}(x)u_{\lambda_0}(y)
\]
where the convergence is locally uniform in $(x,y)\in(0,\infty)\times(0,\infty)$.
\end{theorem}

\begin{proof}
Let us first introduce the splitting by defining the operators $L_a$ and $L^a$ associated to the closure of the quadratic forms $q_a$
\[
C_c^\infty((0,\infty))\ni f\mapsto q_a(f)=\frac{1}{2}\int_a^\infty|f'(t)|^2\rho(dt)
\]
and $q^a$
\[
C_c^\infty((0,\infty))\ni f\mapsto q^a(f)=\frac{1}{2}\int_0^a|f'(t)|^2\rho(dt)
\]
It was shown in \cite{C09} that for every $a\in(0,\infty)$ the spectrum of $L_a$ is discrete (note that by the continuity of $b$ the boundary $a$ is regular). Due to the decomposition principle (see section 131 in \cite{AG54}) it is thus enough to prove that the spectrum of the remaining operator $L^a$ is discrete at $0$. This will be done by using a well-known result of Hartmann (see e.g. Theorem 1.1 in \cite{W67}), which says that the essential spectrum is empty if for every $\lambda\in\R$ every solution of the equation $(\tau-\lambda)u=0$ has only finitely many zeroes in $(0,a)$.

We argue by contradiction. Assume that $\lambda>0$ and that $v$ is a nontrivial solution of  $(\tau-\lambda)u=0$ and that $v$ has infinitely many zeroes in $(0,a)$. Observe that no $x_0$ with $v(x_0)=0$ is a local extremum since $v$ is assumed to be nontrivial solution. Between two successive zeroes there is necessarily a local extremum. As $\lambda>0$ and $v$ satisfies $\tau v=\lambda v$, local maxima of $v$ are necessarily positive and local minima negative. As a result we can pick two sequences $(x_n)_{n\in \N}$ and $(\tilde x_n)_{n\in \N}$ such that
\begin{enumerate}
\item For all $n\in\N$
\[
x_n<\tilde x_n,
\]
\item
\[
\lim_{n\to\infty} x_n=0,
\]
\item For all $n\in\N$ the function $v$ is non-increasing on $(x_n,\tilde x_n)$ and $v(x_n)>0$.
\end{enumerate}
This implies that
\[
0<v(x_n)=\int_{x_n}^{\tilde x_n}\rho(s)^{-1}\,ds\int_{x_n}^s\rho(t)2\lambda v(t)dt\leq 2\lambda v(x_n)\int_{x_n}^{\tilde x_n}\rho(s)^{-1}\,ds\int_{x_n}^s\rho(s)\,dt
\]
and therefore
\begin{equation}\label{e:contra}
\frac{1}{2\lambda}\leq \int_{x_n}^{\tilde x_n}\rho(s)^{-1}\,ds\int_{x_n}^s\rho(s)\,dt.
\end{equation}
Since $0$ is an exit boundary we must have $\int_1^0\rho(s)^{-1}\,ds\int_1^s\rho(t)\,dt<\infty$. This contradicts equation \eqref{e:contra}. Thus every solution of the eigenvalue equation $(\tau-\lambda)u=0$ has only a finite number of zeroes in $(0,a)$.

It remains to prove the integrability of the eigenfunction $v_{\lambda_0}$ corresponding to the lowest eigenvalue $\lambda_0$. First observe that $\lim_{x\to 0}v_{\lambda_0}(x)=0$. This follows from the fact that by the definition of $L$ there exists a sequence $(\varphi_n)_{n\in\N}\subset C_c^\infty((0,\infty))$ which converges to $v_{\lambda_0}$ with respect to the norm
\[
\mathcal{D}(q)\ni f\mapsto \left(\|f\|^2_{L^2((0,\infty),\rho)}+\int_0^\infty|f'(x)|^2\rho(dx)\right)^{\frac{1}{2}}.
\]
The elementary inequality
\[
\sup_{x\in[0,a]}|\varphi(x)|\leq C_a\left(\int_0^\infty|\varphi'(x)|^2\rho(dx)\right)^{\frac{1}{2}}
\]
implies that for $a\in(0,\infty)$, $\|\varphi_n-\varphi_m\|_{C([0,a))}\to 0$ as $n,m\to\infty$ and therefore $\lim_{x\to 0}v_{\lambda_0}(x)=0$.

By Theorem \ref{t:exit_spectral} there exist solutions $u_{\lambda_0}$ and $\tilde u_{\lambda_0}$ of the equation $(\tau -\lambda_0)u=0$ with the following properties
\begin{enumerate}
\item $\lim_{x\to 0}u_{\lambda_0}(x)=1$,
\item $\lim_{x\to 0}\tilde u_{\lambda_0}(x)=0$ and $\tilde u_{\lambda_0}\in L^1((0,a),\rho)$ for all $a\in (0,\infty)$,
\item $(u_{\lambda_0}, \tilde u_{\lambda_0})$ forms a basis for the space of solutions of $(\tau -\lambda_0)u=0$.
\end{enumerate}
We conclude that $v_{\lambda_0}\in L^1((0,a),\rho)$. Since $\infty$ is entrance we also have $\int_1^\infty\rho(dy)<\infty$ which implies
\[
v_{\lambda_0}\in L^1((1,\infty),\rho)\subset L^2((0,\infty),\rho).
\]
Note that we used that $v_{\lambda_0}$ is an $L^2$-eigenfunction. Summing everything up, we have shown that $v_{\lambda_0}\in L^1((0,\infty),\rho)$.

Since we have proved that the spectrum is discrete and that the eigenfunctions are integrable, the spectral representation of $p(t,x,y)$ now follows directly from the spectral theorem. The last assertion of the theorem can be proved exactly as in \cite{S93}:

One first observes that a direct spectral theoretic consequence is the fact that for all $f,g \in L^2((0,\infty),\rho)$ we have that as $t\to\infty$
\begin{displaymath}
e^{\lambda_0t}\left\langle f,\int_0^{\infty}p(t,\cdot,y)g(y)\rho(dy)\right\rangle \rightarrow \langle u_{\lambda_0},f\rangle_{L^2((0,\infty),\rho)}\langle u_{\lambda_0},g\rangle_{L^2((0,\infty),\rho)}.
\end{displaymath}
Then arguing as in \cite{S93} we conclude that $p(t,x,\cdot)\in L^2((0,\infty),\rho)$ and therefore
\begin{displaymath}
p(s+t,x,z)=\left\langle p(s/2,x,\cdot),\int_0^{\infty}p(t,y,z)p(s/2,z,\cdot)\rho(dz)\right\rangle_{L^2((0,\infty),\rho)}.
\end{displaymath}
This concludes the proof.
\end{proof}

\subsection{Positivity Criterion for the Spectrum}\label{s:positivity} The main results are based on the positivity of the spectrum of $L$, that is $\lambda_0>0$. Because this criterion is not very hands-on we include a criterion that is only in terms of the drift $b$. For a regular boundary at $0$ the following criterion was proved in \cite{P09} (also see \cite{M72} for a closely related result). To adjust it to our singular setting we use, just like in previous proofs, a splitting (or decomposition) method.
\pos*
\begin{proof}
If $\lambda_0$ is an isolated eigenvalue, then we have $\lambda_0>0$. This follows from the transience of $L$ and general criticality theory (see section 5.1 in \cite{P95}). A simple proof of the assertion that $0$ cannot be an isolated eigenvalue adapted to our setting uses exactly the same strategy as in assertion vi) of \cite{C09}. Therefore $\lambda_0=0$ if and only if $\lambda_0=\inf \sigma_{ess}(L)=0$ and it remains to prove that the bottom of the essential spectrum of $L$, i.e. the complement of the set of isolated eigenvalues, is positive if and only if $A(b,a)<\infty$ for some $a\in(0,\infty)$.

Observe that by the decomposition principle (see Section 131 in \cite{AG54} and the proof of Theorem \ref{t:exit_entrance_spectrum} the essential spectrums of the operators $L$ and $L_a$ coincide. Note that due to our assumptions the boundary point $a$ is regular for the $L_a$-diffusion and absorption at $a$ is certain for the $L_a$-diffusion. In Theorem 1 equation (1.11) from \cite{P09} it is shown that
\[
\frac{1}{8A(b,a)}\leq \inf \sigma(L_a)\leq \frac{1}{2A(b,a)}.
\]

Thus, if $A(b,a)<\infty$ for some $a\in(0,\infty)$, then
\begin{displaymath}
\inf \sigma_{ess}(L)= \inf \sigma_{ess}(L_a) \geq \inf \sigma(L_a)>0
\end{displaymath}
On the other hand if $A(b,a)=\infty$ for some $a\in(0,\infty)$, then we have
\begin{displaymath}
0=\inf\sigma_{ess}(L_a)=\inf\sigma_{ess}(L) \geq \inf\sigma(L) \geq 0.
\end{displaymath}
\end{proof}

Of course, if one does not assume that absorption is certain it might happen that $\lambda_0>0$ and $\int_1^\infty \rho(y)\,dy=\infty$. We will see that in such a situation mass escapes towards infinity with an exponential rate given by $\lambda_0$.

\section{Yaglom Limits}\label{s:Yaglom}
This section is dedicated to proofs of the main existence results. We make use heavily of the machinery developed in Section \ref{s:analytic}. The section is structured according to the boundary classification at $\infty$. We start with $\infty$ being a natural boundary and then consider entrance and exit boundaries.

\subsection{Natural boundary at $\infty$}\label{s:natural} For a natural boundary at $\infty$ the main strategy is as follows: First, we prove that a local form of convergence to a quasistationary distribution holds by restricting to compact sets which, in the notation of \cite{ES07}, is called convergence to a quasistationary distribution on compacta. In a second step, we prove that either (non-local) convergence to a quasistationary distribution or escape to infinity occurs.

\subsubsection{Convergence on Compacta} We show how one can use spectral calculus to derive convergence to the quasistationary distribution $\vp(\lambda_,\cdot)$ for compact sets.

\begin{theorem}\label{t:compact}
Suppose that $0$ is an exit boundary, $\infty$ is natural and Assumption \ref{a:1} holds. If $A\subset B$ are compact subsets of $[0,\infty)$ and $\nu$ is a compactly supported initial distribution on $(0,\infty)$, then
\[
\lim_{t\to\infty} \Pr^\nu(X_t\in A~|~X_t\in B) = \frac{\int_A\vp(\lambda_0,x)\rho(dx)}{\int_B\vp(\lambda_0,x)\rho(dx)},
\]
where $\vp(\lambda_0,\cdot)$ is the generalized eigenfunction from Definition \ref{d:vp}
\end{theorem}

We start by sketching how to use the general form of the spectral theorem (Theorem \ref{t:exit_spectral}) to find the Yaglom limit. First, to get from probability to spectral analysis note that
\begin{equation*}
\begin{split}
\Pr^x(X_t\in A~|~X_t\in B) &=\frac{\Pr^x(X_t\in A)}{\Pr^x(X_t\in B)} = \frac{e^{-tL}\ind_A(x)}{e^{-tL}\ind_B(x)}\\
&=\frac{E[[\lambda_0,\lambda_1]]e^{-tL}\ind_A(x)+E[(\lambda_1,\infty)]e^{-tL}\ind_A(x)}{E[[\lambda_0,\lambda_1]]e^{-tL}\ind_B(x)+E[(\lambda_1,\infty)]e^{-tL}\ind_B(x)}
\end{split}
\end{equation*}
To make the latter rigourous we need to impose Assumption \ref{a:1} as $\ind_A$ does not need to be square integrable with respect to $\rho(x)$. To show that the right hand side of the last equation converges to
\[
\frac{\int_A\vp(\lambda_0,x)\,\rho(dx)}{\int_B\vp(\lambda_0,x)\,\rho(dx)},
\]
we show that the part determined by the spectrum away from its bottom $\lambda_0$ is negligible, i.e. $E[(\lambda_1,\infty)]e^{-tL}\ind_B(x)$ and $E[(\lambda_1,\infty)]e^{-tL}\ind_A(x)$ converge to zero fast enough. The spectral part at the bottom determines the limit, as we show that
\begin{equation}\label{e:conv_bottom}
\lim_{t\to \infty}\frac{E([\lambda_0,\lambda_1])e^{-tL}\ind_A(x)}{E([\lambda_0,\lambda_1])e^{-tL}\ind_B(x)}= \frac{\int_A\vp(\lambda_0,x)\,\rho(dx)}{\int_B\vp(\lambda_0,x)\,\rho(dx)}.
\end{equation}
In this step the spectral theorem plays a significant role as by Theorem \ref{t:exit_spectral} v) one has
\[
E([\lambda_0,\lambda_1])e^{-tL}\ind_A(x) = \int_{\lambda_0}^{\lambda_1}e^{-t\lambda}\vp(\lambda,x)\int_A \vp(\lambda,y)\rho(dy)\sigma(d\lambda).
\]
As we have some intuition for what the correct limit is we now carry on the rigorous proof.

\begin{proof}
First note that the generalized eigenfunctions of $L$ exist due to Theorem \ref{t:exit_spectral} . We first show that for $z>0$ and subsets $A,B\subset (0,z]$ with $A\subset B$ we have
\[
\lim_{t\to\infty} \Pr^\nu(X_t\in A~|~X_t\in B) = \frac{\int_A\vp(\lambda_0,x)\rho(dx)}{\int_B\vp(\lambda_0,x)\rho(dx)},
\]
as sketched above. This can be done as in \cite{C09} (see also \cite{ES07}) with some additional work in order to avoid the complication that $\ind_A$ does not necessarily belong to $L^2((0,\infty),\rho)$, which causes the problem that $e^{-tL}$ as an operator in $L^2$ cannot be applied to the functions $\ind_A$ and $\ind_B$. Assumption \ref{a:1} will be used to circumvent this issue.

For a subset $A\subset [0,z]$ and $\ve>0$ define the set $A_\ve:=\{x\in A:|x|>\ve\}$. By Assumption \ref{a:1}, $\lim_{\ve\to 0}e^{-tL}\ind_{A_\ve}=e^{-tL}\ind_A$ in $L^2((0,\infty),\rho)$ for any $t>0$. This allows us to decompose $e^{-tL}\ind_A$ as
\[
e^{-tL}\ind_A(x) = E([\lambda_0,\lambda_1])e^{-tL}\ind_A(x)+ E((\lambda_1,\infty))e^{-tL}\ind_A(x)
\]
in $L^2$ and hence for almost every $x\in(0,\infty)$ which by continuity extends to every $x\in(0,\infty)$. Observe moreover that as sketched above, Theorem \ref{t:exit_spectral} iv) yields for $\lambda_1\in(0,\infty)$ that the continuous integral kernel $h^{\lambda_1}(t,x,y)$ of the operator $E((\lambda_0,\lambda_1))e^{-tL}$ is given by
\[
h^{\lambda_1}(t,x,y) = \int_{[\lambda_0,\lambda_1]} e^{-t\lambda} \vp(\lambda,x)\vp(\lambda,y)\sigma(d\lambda).
\]
We will sometimes simplify notation and write $\langle,\rangle$ instead of $\langle,\rangle_{L^2((0,\infty),\rho)}$.

If $\nu$ is an initial distribution with compact support in $(0,\infty)$, then due to the continuity of the functions $e^{-tL}\ind_A, E([\lambda_0,\lambda_1])e^{-tL}\ind_A$, and $E((\lambda_1,\infty))e^{-tL}\ind_A$ the expressions
\[
\langle e^{-tL}\ind_A, \nu \rangle, ~\langle E([\lambda_0,\lambda_1])e^{-tL}\ind_A, \nu \rangle~\text{and}~~\langle E((\lambda_1,\infty))e^{-tL}\ind_A, \nu \rangle
\]
are well-defined and we get
\[
\langle e^{-tL}\ind_A, \nu \rangle = \langle E([\lambda_0,\lambda_1])e^{-tL}\ind_A, \nu \rangle + \langle E((\lambda_1,\infty))e^{-tL}\ind_A, \nu \rangle
\]
Using Fubini's theorem we see that
\begin{equation*}
\begin{split}
\langle E([\lambda_0,\lambda_1])e^{-tL}\ind_A(\cdot),\nu\rangle &= \left\langle \int_0^\infty h^{\lambda_1}(t,\cdot,y)\ind_A\rho(dy),\nu \right\rangle\\
&=\int_0^\infty\int_{[\lambda_0,\lambda_1]} e^{-\lambda t}\vp(\lambda,x)\int_A\vp(\lambda,y)\rho(dy)\sigma(d\lambda)\nu(dx)\\
&=\int_{[\lambda_0,\lambda_1]} e^{-\lambda t} \int_0^\infty \vp(\lambda,x) \nu(dx) \int_A \vp(\lambda,y)\rho(dy)\sigma(d\lambda).
\end{split}
\end{equation*}
The use of Fubini's theorem can be justified by using the properties of the generalized eigenfunctions $\vp$. Now observe that, just as in the regular case (see \cite{ES07} and \cite{C09})
\begin{equation}\label{e:limit_spectral}
\lim_{t\to \infty} \frac{\left\langle E([\lambda_0,\lambda_1])e^{-tL}\ind_A,\nu\right\rangle}{\left\langle E([\lambda_0,\lambda_1])e^{-tL}\ind_{[0,z]},\nu \right\rangle} = \frac{\int_A\vp(\lambda_0,x)\rho(dx)}{\int_0^z\vp(\lambda_0,x)\rho(dx)}
\end{equation}
We are left to show that the remainder terms corresponding to the spectrum away from its bottom are negligible, i.e.
\[
\lim_{t\to\infty}\Pr^\nu(X_t\in A~|~X_t\leq z)=\lim_{t\to \infty} \frac{\left\langle E([\lambda_0,\lambda_1])e^{-tL}\ind_A,\nu\right\rangle}{\left\langle E([\lambda_0,\lambda_1])e^{-tL}\ind_{[0,z]},\nu \right\rangle} .
\]
We use ideas that are similar to the ones used in the case of a regular boundary at $0$. Observe that
\begin{equation*}
\begin{split}
\frac{\Pr^\nu(X_t\in A)}{\langle E([\lambda_0,\lambda_1])e^{-tL}\ind_A , \nu \rangle} &= \frac{\langle E([\lambda_0,\lambda_1])e^{-tL}\ind_A, \nu \rangle+ \langle E((\lambda_1,\infty))e^{-tL}\ind_A, \nu \rangle}{\langle E([\lambda_0,\lambda_1])e^{-tL}\ind_A, \nu \rangle}\\
&= 1 + \frac{\langle E((\lambda_1,\infty))e^{-tL}\ind_A, \nu \rangle}{\langle E([\lambda_0,\lambda_1])e^{-tL}\ind_A, \nu \rangle}
\end{split}
\end{equation*}
Using $e^{-tL}\ind_A\in L^2((0,\infty),\rho)$ and the elementary inequality
\[
|g(x)| \leq \left(\int_0^x\rho(x)^{-1}dx\right)^{\frac{1}{2}}\left(\int_0^\infty|g'(x)|^2\rho(dx)\right)^{\frac{1}{2}}, ~~g\in \mathcal{D}(\sqrt{L})
\]
yields
\begin{equation*}
\begin{split}
\langle E((\lambda_1,\infty))e^{-tL}\ind_A,\nu\rangle &\leq C_\nu \left(\int_0^\infty|g'_t(x)|^2\rho(dx)\right)^{\frac{1}{2}}\\
&=\left(\int_{\lambda_1}^\infty e^{-2(t-\ve)\lambda}\|E(d\lambda)e^{-\ve L}\ind_A\|_{L^2((0,\infty),\rho)}\right)^{\frac{1}{2}}
\end{split}
\end{equation*}
where $g_t=E((\lambda_1,\infty))e^{-tL}\ind_A$ and $\ve>0$ is small enough. As a result the map
\[
t\mapsto \langle E((\lambda_1,\infty))e^{-tL}\ind_A,\nu \rangle
\]
decays exponentially with an exponential rate which is strictly larger than the exponential rate of decay of
\[
t\mapsto \langle E([\lambda_0,\lambda_1])e^{-tL}\ind_A,\nu \rangle.
\]
We can now finish the proof noting that
\begin{equation*}
\begin{split}
\lim_{t\to \infty}\Pr^\nu(X_t\in A~|~X_t\in B) &= \lim_{t\to \infty} \frac{\Pr^\nu(X_t\in A)}{\Pr^\nu(X_t\in B)}\\
&= \lim_{t\to \infty} \frac{E[[\lambda_0,\lambda_1]]e^{-tL}\ind_A(x)+E[(\lambda_1,\infty)]e^{-tL}\ind_A(x)}{E[[\lambda_0,\lambda_1]]e^{-tL}\ind_B(x)+E[(\lambda_1,\infty)]e^{-tL}\ind_B(x)}\\
&= \lim_{t\to \infty} \frac{E[[\lambda_0,\lambda_1]]e^{-tL}\ind_A(x)}{E[[\lambda_0,\lambda_1]]e^{-tL}\ind_B(x)}\\
&= \frac{\int_A\vp(\lambda_0,x)\rho(dx)}{\int_B\vp(\lambda_0,x)\rho(dx)}.
\end{split}
\end{equation*}
\end{proof}

Having proved the existence of the Yaglom limit for compact sets $A$, we now have a good guess as to which quasistationary distribution is the correct limit. All we have to do is extend the reasoning from compact $A$ to arbitrary $A$.

\subsubsection{Proof of Theorem \ref{t:exit_natural}} We now present an elegant argument, which in the case when $0$ is regular, is due to Steinsaltz and Evans \cite{ES07}.

\begin{lemma}\label{l:3}
Suppose $\infty$ is a natural boundary point, the initial distribution $\nu$ is compactly supported, and $\lambda_0>0$. Then one of the following two properties holds
\begin{enumerate}
\item[i)] $(X_t)_{t\geq 0}$ converges to the quasistationary distribution $\vp(\lambda_0,\cdot)$.
\item[ii)] $(X_t)_{t\geq 0}$ escapes to infinity.
\end{enumerate}
\end{lemma}
\begin{proof}
Let
\[
f(z,t):=\frac{\Pr^\nu(X_t>z)}{\Pr^\nu(X_t\leq z)},
\]
and note that
\begin{equation*}
\begin{split}
\lim_{t\to \infty}f(z,t)=\infty,  z\geq 0 &\iff (X_t)_{t\geq 0} ~\text{escapes to infinity},\\
\lim_{z\to\infty}\lim_{t\to \infty}f(z,t)=0,  z\geq 0 &\iff (X_t)_{t\geq 0} ~\text{converges to the quasistationary distribution}~\vp(\lambda_0,\cdot).
\end{split}
\end{equation*}
In order to prove the dichotomoy one has to make sure that no other limits of $f$ other than $0$ and $\infty$ are possible. First note that for fixed $0\leq a\leq z$
\begin{equation*}
\begin{split}
\Pr^\nu(X_{n+1}>z) &=\Pr^\nu(X_{n+1}\geq a) - \Pr^\nu(a\leq X_{n+1}\leq z)\\
&\geq \Pr^\nu(X_{n+1}\geq a~|~X_n>z)\Pr^\nu(X_n>z) - \Pr^\nu(a\leq X_{n+1}\leq z)\\
&\geq \Pr^z(\forall t\in[0,1]:X_t\geq a)\Pr^\nu(X_n>z) - \Pr^\nu(a\leq X_{n+1}\leq z)
\end{split}
\end{equation*}
where we used the strong Markov property. Due to $\infty$ being natural, $\Pr^z(\forall t\in[0,1]:X_t\geq a)$ converges to $1$ as $z\to \infty$. Further, if we use the convergence to the quasistationary distribution $\vp(\lambda_0,\cdot)$ on compacta from Theorem \ref{t:compact} we get
\[
\lim_{n\to\infty}\frac{\Pr^\nu(a\leq X_{n+1}\leq z)}{\Pr^\nu( X_{n+1}\leq z)} = \frac{\int_a^z\vp(\lambda_0,x)\rho(dx)}{\int_0^z\vp(\lambda_0,x)\rho(dx)}\leq\frac{\int_a^\infty\vp(\lambda_0,x)\rho(dx)}{\int_0^z\vp(\lambda_0,x)\rho(dx)}.
\]
For each $\ve>0$ we can find $n', z_0,a$ such that for all $n\geq n'$ and $z\geq z_0$
\[
\Pr^z(\forall t\in[0,1]:X_t\geq a) > 1-\ve
\]
and
\[
\frac{\Pr^\nu(a\leq X_{n+1}\leq z)}{\Pr^\nu( X_{n+1}\leq z)}<\ve.
\]
Since $\lambda_0>0$ we can find some $n''$ such that for $n\geq n''$ and $z\geq z_0$ and $\ve$ small enough
\[
\frac{\Pr^\nu( X_{n}\leq z)}{\Pr^\nu( X_{n+1}\leq z)}\geq q>1/(1-\ve).
\]

We get that if $a,n,z$ are large enough then

\begin{equation*}
\begin{split}
f(z,n+1) &= \frac{\Pr^\nu(X_{n+1}>z)}{\Pr^\nu(X_{n+1}\leq z)}\\
&\geq \frac{\Pr^z(\forall t\in[0,1]:X_t\geq a)\Pr^\nu(X_n>z)}{\Pr^\nu(X_{n+1}\leq z)} - \frac{\Pr^\nu(a\leq X_{n+1}\leq z)}{\Pr^\nu(X_{n+1}\leq z)}\\
&\geq \Pr^z(\forall t\in[0,1]:X_t\geq a)\frac{\Pr^\nu(X_{n}> z)}{\Pr^\nu(X_{n}\leq z)}\frac{\Pr^\nu(X_{n}\leq z)}{\Pr^\nu(X_{n+1}\leq z)} - \frac{\Pr^\nu(a\leq X_{n+1}\leq z)}{\Pr^\nu(X_{n+1}\leq z)}\\
&\geq q(1-\ve)f(z,n)-\ve.
\end{split}
\end{equation*}
Since $\ve$ is arbitrarily small, taking limits on both sides yields
\begin{itemize}
\item $\limsup _{n\to \infty}f(n,z)=\infty, \forall z\geq 0$; or
\item $\lim_{z\to \infty}\limsup _{n\to \infty}f(n,z)=0$.
\end{itemize}

We still have to extend the above to real times $t\geq 0$. First note that
\begin{equation*}
\begin{split}
&\lim_{z\to \infty}\limsup_{t\to \infty} \frac{\Pr^\nu(X_t>z)}{\Pr^\nu(X_t\leq z)}\\
&= \lim_{z\to \infty}\limsup_{t\to \infty} \frac{\Pr^\nu(X_t>z)}{\Pr^\nu(X_n\leq z)}\frac{\Pr^\nu(X_n\leq z)}{\Pr^\nu(X_t\leq z)}\\
&\leq  \lim_{z\to \infty}\limsup_{n\to \infty}  \frac{\Pr^\nu(\exists t\in[n,n+1): X_t>z)}{\Pr^\nu(X_n\leq z)}\sup_{t\in[n,n+1)}\frac{\Pr^\nu(X_n\leq z)}{\Pr^\nu(X_t\leq z)}.
\end{split}
\end{equation*}
We also have
\begin{equation*}
\begin{split}
&\Pr^\nu(\exists t\in[n,n+1): X_t>z)\\
&=\Pr^\nu(\exists t\in[n,n+1): X_t>z, X_n\geq z) + \Pr^\nu(\exists t\in[n,n+1): X_t>z, X_n< z)\\
&\leq \Pr^\nu(X_n\geq z) + \Pr^\nu(\exists t\in[n,n+1): X_t>z~|~ X_n< z)\Pr^\nu(X_n\leq z)\\
&\leq \Pr^\nu(X_n\geq z) + \Pr^\nu(\exists t\in[n,n+1): X_t>z~|~ X_n< z).
\end{split}
\end{equation*}
The second summand tends to $0$ due to the convergence to the quasistationary distribution $\vp(\lambda_0,\cdot)$ on compacta. One can see that
\[
\lim_{n\to\infty}\Pr^\nu(\exists t\in[n,n+1): X_t>z~|~ X_n< z) = \frac{\int_0^z\Pr^x(\exists t\in[0,1): X_t>z)\vp(\lambda_0,x)\rho(dx) }{\int_0^\infty \vp(\lambda_0,x)\rho(dx)}.
\]
Without loss of generality we may assume $\int_0^\infty \vp(\lambda_0,x)\rho(dx)<\infty$ since otherwise $(X_t)_{t\geq 0}$ escapes to infinity. As a result, due to dominated convergence the limit in $z$ can be taken inside and we get $0$ since $\infty$ is inaccessible. We therefore get
\[
\lim_{z\to \infty}\lim_{t\to \infty}\frac{\Pr^\nu(X_t>z)}{\Pr^\nu(X_t\leq z)}=0.
\]

The second case is similar. For any $a>z$
\begin{equation*}
\begin{split}
&\liminf_{t\to\infty}f(z,t)\\
&\geq\liminf_{n\to\infty}\frac{\Pr^\nu(X_n>a)-\Pr^\nu(a<X_n, \exists t\in [n,n+1):X_t\leq z)}{\Pr^\nu(X_n\leq a)+\Pr^\nu(a<X_n, \exists t\in [n,n+1):X_t\leq z)}\\
&\geq\liminf_{n\to\infty} \frac{f(a,n)(1-\Pr^\nu(\exists t\in [n,n+1):X_t\leq z~|~a<X_n))}{1+f(a,n)\Pr^\nu(\exists t\in [n,n+1):X_t\leq z~|~a<X_n)}\\
&=\liminf_{n\to\infty} \Pr^\nu(\exists t\in [n,n+1):X_t\leq z~|~a<X_n)^{-1} -1
\end{split}
\end{equation*}
where the last equality is true since $\lim_{n\to \infty} f(a,n)=\infty$. Hence, the right hand side diverges as $a$ tends to infinity since $\infty$ is natural. We thus showed that for all $z\geq 0$
\[
\lim_{t\to \infty}f(z,t)=\infty.
\]
\end{proof}
\begin{remark}
We highlight two important features of Lemma \ref{l:3}. First, we obtained a dichotomy of escape to infinity and convergence to a quasistationary distribution which, thus far, does not give any criterion that helps to decide. Second, the initial distribution $\nu$ is fixed, whereas the main theorems are formulated for arbitrary compactly supported initial conditions.
As shown in the next proposition the important observation is that the dichotomy does lead to a precise criterion: true exponential decrease of survival probabilities, i.e.
\[
\lim_{t\to \infty}\frac{1}{t}\log \Pr^\nu(T_0>t)<0
\]
turns out to imply convergence to the quasistatioanry distribution $\vp(\lambda_0,\cdot)$.
\end{remark}

\begin{proposition}\label{p:exp_decrease}
Suppose $0$ is an exit boundary, $\infty$ is natural, absorption is certain, $\nu$ is a compactly supported initial distribution, and $\lambda_0>0$. Then
\[
\lim_{t\to \infty}\frac{1}{t}\log \Pr^\nu(T_0>t)<0
\]
implies the convergence to the quasistationary distribution $\vp(\lambda_0,\cdot)$.
\end{proposition}
\begin{proof}
It suffices to show that escape to infinity implies non-exponential decrease of the survival probability. This is true since then exponential decrease implies no escape to infinity which due to the dichotomy from Lemma \ref{l:3} implies convergence to the quasistationary distribution $\vp(\lambda_0,\cdot)$.

Note that escape to infinity implies the weak convergence of the measures $\Pr^\nu(X_t\in\cdot~|~T_0>t)$ to $\delta_\infty$. We apply this result to $f(x)=\Pr^x(T_0>t)$ (the continuity of $f$ follows as in the proof of Lemma 7.2 of \cite{C09}). Combined with the Chapman-Kolmogorov relation this yields
\[
\lim_{s\to\infty}\Pr^\nu(T_0>t+s~|~T_0>t) =\lim_{s\to\infty}\int_0^\infty\Pr^y(T_0>t)\Pr^\nu(X_s\in dy~|~T_0>s)=\lim_{y\to \infty}\Pr^y(T_0>t).
\]
The right hand side equals $1$ since $\infty$ is a natural boundary point. Finally, since
\[
\Pr^\nu(T_0>t+s~|~T_0>t)  = \frac{\Pr^\nu(T_0>t+s)}{\Pr^\nu(T_0>s)}
\]
the survival probabilities cannot decrease exponentially fast.
\end{proof}

We are now ready to prove our main result
\natural*

In particular, Proposition \ref{p:exp_decrease} shows that in order to prove Theorem \ref{t:exit_natural}, it suffices to check that $\lambda_0>0$ implies the exponential decay of survival probabilities for some compactly supported initial distribution. This can be done using a martingale argument.
\begin{proof}
By Proposition \ref{p:exp_decrease} it is enough to show that for every initial distribution $\nu$ with compact support in $(0,\infty)$
\[
\lim_{t\to \infty}\frac{1}{t}\log \Pr^\nu(T_0>t)<0.
\]
The bottom of the spectrum of the operator $L-\lambda_0/2$ is $\lambda_0/2$ which implies that the operator $L-\lambda_0/2$ is subcritical or equivalently there exit two linearly independent positive solutions $u_1, u_2$ of the equation $(\tau-\lambda_0/2)u=0$ (see for example Proposition 1.2 in \cite{P95}). Thus, by assertion i) of Theorem \ref{t:exit_spectral} there exists a positive solution $u$ of $(\tau-\lambda_0/2)u=0$ with $\lim_{x\to 0+}u(x)=1$.

Consider the stochastic process defined by $Y_t:=e^{\frac{\lambda_0}{2}t\wedge T_0}u(X_{t\wedge T_0}), t\geq 0$. Using Ito's formula one sees that that $(Y_{t\wedge T_M})_{t\geq 0}$ is a martingale with respect to $\Pr^x$ for all $M>0$. This forces for all $x\in(0,\infty)$
\[
\E^x\left[e^{\frac{\lambda_0}{2}t\wedge T_0\wedge T_M}u(X_{t\wedge T_0\wedge T_M})\right] = u(x).
\]
We let $M\to \infty$ and $t\to\infty$ and apply Fatou's lemma twice to get
\[
\E^x\left[e^{\frac{\lambda_0}{2}T_0}\right]\leq u(x).
\]
Now we integrate with respect to the initial distribution $\nu$, use the exponential Markov inequality and the integrability of $e^{\frac{\lambda_0}{2}T_0}$ for all $x>0$ to get that
\[
\lim_{t\to \infty}\frac{1}{t}\log \Pr^\nu(T_0>t)\leq \lim_{t\to \infty}\frac{1}{t}\log \frac{\E^\nu\left[e^{\frac{\lambda_0}{2}T_0}\right]}{e^{\frac{\lambda_0}{2}t}}\leq -\frac{\lambda_0}{2}
\]
which is strictly negative by assumption. Thus in the case of certain absorption the condition $\lambda_0>0$ implies the convergence to the Yaglom limit independently of the initial distribution $\nu$.

We next show that if absorption is not certain, $(X_t)_{t\geq 0}$ escapes exponentially fast to infinity, that is, for all $a\in (0,\infty)$ and every compact subset $K\subset (0,\infty)$
\[
\limsup_{t\to \infty} \sup_{x\in (0,a)} e^{\lambda_0 t}\Pr^x(X_t\in K~|~T_0>t)<\infty.
\]

Note, that as we mentioned before, for any fixed $a>0$ there exists a constant $C_a>0$ such that for all $f\in \mathcal{D}(q)$
\[
\sup_{x\in (0,a)}|f(x)|\leq C_a \left(\int_0^\infty |f'(t)|^2\rho(dt)\right)^{\frac{1}{2}}.
\]
By spectral calculus we conclude that $e^{-tL}g\in \mathcal{D}(q)$ for every $g\in L^2((0,\infty),\rho)$ and therefore

\begin{equation*}
\begin{split}
\sup_{x\in (0,a)}\Pr^x(X_t\in K, T_0>t) &= \sup_{x\in (0,a)} (e^{-tL}\ind_K)(x)\\
&\leq  C_a \left(\int_0^\infty |(e^{-tL}\ind_K)'(x)|^2\rho(dt)\right)^{\frac{1}{2}}\\
&= C_a\left(\int_{\lambda_0}^\infty e^{-2t\lambda}d\|E_\lambda\ind_K\|^2_{L^2((0,\infty),\rho)}\right)
\end{split}
\end{equation*}
We obtain
\[
\limsup_{t\to \infty} \sup_{x\in (0,a)} e^{\lambda_0 t}\Pr^x(X_t\in K,T_0>t)\leq C_a \limsup_{t\to\infty} \left(\int_{\lambda_0}^\infty e^{2\lambda_0 t}e^{-2t\lambda}d\|E_\lambda\ind_K\|^2_{L^2((0,\infty),\rho)}\right).
\]
The right hand side is finite and even equal to $0$ if $\lambda_0$ does not belong to the point spectrum of $L$. Since $\nu$ is compactly supported we have
\[
\Pr^\nu(X_t\in K,T_0>t) = \int _0^\infty \Pr^x(X_t\in K,T_0>t) \nu(dx)
\]
which extends the exponential decay to the initial distribution $\nu$. Finally, since $\Pr^\nu(T_0>t)$ is bounded away from zero, we obtain the exponential decay of
\[
\Pr^\nu(X_t\in K~|~T_0>t)= \frac{\Pr^\nu(X_t\in K,T_0>t)}{\Pr^\nu(T_0>t)}
\]
\end{proof}

\begin{remark}\label{r:doob}
We discuss how one can circumvent the assumption of certain absorption. Since $0$ is an exit boundary, $(X_t)_{t\geq 0}$ hits zero in finite time with positive probability and we can condition on eventual absorption:
\[
\Pr^x(X_t\in\cdot~|~T_0<\infty).
\]
This way one includes the condition that absorption happens in finite time. We can then reduce this problem to ours by using a Doob $h$-transform. The function $h(x)=\Pr^x(T_0<\infty)$ is harmonic and by the general theory of $h$-transforms (see \cite{P95} and Section 4.4 in \cite{C09}) the process $(X_t)_{t\geq 0}$ conditioned to hit $0$ corresponds to the generator $L^h$ whose action is given by
\[
\tau^h f = \frac{1}{h}\tau(hf) = -\frac{1}{2}\frac{d^2}{dx^2}f +\left(b-\frac{h'}{h}\right)\frac{d}{dx}f.
\]
The process associated to the operator $L^h$ can again be defined by Dirichlet form techniques. The associated measures on path space are given by the family $(\Qr^x)_{x\in (0,\infty)}$ that is defined via
\[
\Qr^x(\cdot) = \Pr^x(\cdot ~|~T_0<\infty).
\]
The operator $L^h$ can be realized as a selfadjoint operator in the Hilbert space $L^2((0,\infty,h(x)\rho(dx))$.

Since the spectrum is invariant under $h$-transforms, we conclude that the positivity of the bottom of the spectrum of $L$ implies the positivity of the spectrum of $L^h$.

Note that conditioning on hitting $0$ implies that $0$ is accessible and absorption is certain. Hence, if $0$ is an exit boundary with respect to the measure $\Qr^x$ we can apply our results (if $0$ is regular use the results of \cite{KS12}) in order to conclude that for every Borel set $A\subset (0,\infty)$
\[
\lim_{t\to\infty} \Qr^x(X_t\in A~|~T_0>t)  = \frac{\int_A\vp^h(\lambda_0,x)h(x)\rho(dx)}{\int_0^\infty\vp^h(\lambda_0,x)h(x)\rho(dx)}
\]
where $\vp^h(\lambda_0,x)$ is the unique solution of $(\tau^h-\lambda_0)u=0$ satisfying $\int_0^a|u(x)|^2h(x)\rho^h(dx)<\infty$ for $a\in (0,\infty)$, where $\rho^h$ is the speed measure for the drift $b-\frac{h'}{h}$.
\end{remark}

\begin{remark}\label{r:killing}
We want to point out that our methods allow to include killing in the interior of $(0,\infty)$. Let $\kappa:(0,\infty)\to (0,\infty)$ be the killing rate and define the operator $\tau_\kappa$ via
\[
\tau_\kappa := -\frac{1}{2}\frac{d^2}{dx^2} + b(\cdot)\frac{d}{dx} + \kappa(\cdot).
\]
This corresponds to a regular diffusion $(X_t)_{t\geq 0}$ killed in the interior of $(0,\infty)$ using the rate $\kappa$. If $U$ is a mean one exponential random variable that is independent of the diffusion $X$ we define the killing time to be
\begin{equation*}\label{e_Tkilling}
T_{\kappa}:=\inf \left\{t>0:\int_0^t \kappa(X_s)\,ds >U\right\}.
\end{equation*}
and we set
\begin{equation*}\label{e:tpartial}
T_\partial:=\min(T_0,T_\kappa).
\end{equation*}
Then one can define the concepts of quasistationary distribution and Yaglom limit using the random time $T_\partial$ instead of $T_0$ in Definitions \ref{d:qsd_1} and \ref{d:Yaglom}.

Our process can in this case die due to the internal killing and not only due to the absorption at $0$. It therefore makes sense to also look at the case when both $0$ and $\infty$ are inaccessible (whereas until now we always wanted $0$ to be accessible).
In order to prove the analogues of our results with internal killing one has to follow the steps which have been worked out in detail in \cite{KS12}. See \cite{CV152} for when one has internal killing and $0$ is exit or regular and $\infty$ is entrance.

\end{remark}

\subsubsection{Sufficient Conditions for Assumption \ref{a:1}}\label{s:assumption} We next give some necessary conditions for Assumption \ref{a:1}, Unfortunately, these conditions are not minimal. A complete understanding of Assumption \ref{a:1} remains open. Furthermore, it is very natural to conjecture that Assumption \ref{a:1} is not necessary in the proof of Theorem \ref{t:exit_natural}.
\assumption*

\begin{proof}
We use several ideas from \cite{C09}. First observe that it is enough to prove that $e^{-tL}\ind_{[0,\ve]}\in L^2((0,\infty),\rho)$ for $0<\ve<1$ as for every $0<\ve<z$ the function $\ind_{(\ve,z)}$ obviously belongs to $L^2(\rho)$. By Girsanov's theorem (see Proposition 2.2 of \cite{C09}) and our assumptions we have

\begin{equation}\label{e:heat_ineq}
\begin{split}
e^{-tL}\ind_{[0,\ve]}(x) &= \E^x\left[\ind_{[0,\ve]}(B_t)e^{\frac{1}{2}Q(x)-\frac{1}{2}Q(B_t)-\frac{1}{2}\int_0^t(b^2-b')(B_s)\,ds},t<T_0\right]\\
&\leq e^{-\frac{c}{2}t}e^{\frac{1}{2}Q(x)}\int_0^\ve p^D(t,x,y)e^{-\frac{1}{2}Q(y)}dy,
\end{split}
\end{equation}
where $(B_t)_{t\geq 0}$ is a standard Brownian motion, $Q(x):=2\int_1^x b(s)\,ds$ and $p^D(t,x,y)$ denotes the heat kernel for the Laplacian on $(0,\infty)$ with Dirichlet boundary conditions at $0$, i.e.
\[
p^D(t,x,y) = \frac{1}{\sqrt{2\pi t}}\left(e^{\frac{(x-y)^2}{2t}}-e^{\frac{(x+y)^2}{2t}}\right)=\frac{\sqrt{2}}{\sqrt{\pi t}} e^{\frac{-x^2}{2t}}e^{\frac{-y^2}{2t}}\sinh\left(\frac{xy}{t}\right)
\]
We obtain from \eqref{e:heat_ineq} that
\begin{equation}\label{e:heat_ineq2}
\begin{split}
e^{-tL}\ind_{[0,\ve]}(x) \leq e^{-\frac{c}{2}t}e^{\frac{1}{2}Q(x)}\frac{\sqrt{2}}{\sqrt{\pi t}} e^{\frac{-x^2}{2t}} e^{\frac{x}{t}} \int_0^\ve ye^{-\frac{1}{2}Q(y)}\,dy,
\end{split}
\end{equation}
where we used the convexity of $\sinh$ to get $\sinh(\frac{xy}{t})\leq y\sinh(\frac{x}{t})\leq \frac{y}{2}e^{\frac{x}{t}}$. Finally, equation \eqref{e:heat_ineq2} implies that $e^{-tL}\ind_{[0,\ve]}\in L^2((0,\infty),\rho)$ as well as $\lim_{\ve\to 0}\left\|e^{-tL}\ind_{[0,\ve]}\right\|_{L^2((0,\infty),\rho)}=0$.

\end{proof}

\subsection{Applications to Bessel processes}\label{s:bessel}

We apply our results to Bessel processes, i.e. diffusions associated to
\[
\tau^\nu = -\frac{1}{2}\frac{d^2}{dx^2} -\frac{2\nu+1}{2x}\frac{d}{dx},
\]
where $\nu\in \R$ (see \cite{BS12} for basic facts). For $\nu\leq -1$ (the drift is pointing strongly enough to $0$) it is known that $0$ is an exit boundary -- we will make use of this fact. We will consider regular perturbations
\[
\tau^{\nu,c} =\tau^\nu +c(x)\frac{d}{dx}
\]
of Bessel generators and give sufficient conditions for Assumption \ref{a:1} to be satisfied. In particular we show tha the case $c\equiv 0$ and the examples given in Remark 4.6 of \cite{C09} fulfill Assumption \ref{a:1}.

\begin{proposition}\label{p:bessel}
Consider the diffusion associated to the generator
\[
\tau^{\nu,c} =\tau^\nu +c(x)\frac{d}{dx}.
\]
Assume that $\nu\leq -1$ and $c\in C^1((0,\infty))$ is such that $\inf_{s>0}(c(s)^2-c'(s))>-\infty$ and
\[
\inf_{s\geq 1} \frac{c(s)}{s}>-\infty.
\]
If, in addition, $\Pr^x(T_0<\infty)=1$ for all $x\in(0,\infty)$, then Assumption \ref{a:1} holds.
\end{proposition}
\begin{proof}
We have to show, that for $\ve >0$ the function $e^{-tL}\ind_{[0,\ve]}$ belongs to $L^2$ with respect to the speed measure of the diffusion. Without loss of generality we can assume that $\ve<1$.

The strategy for the proof is as follows. First, we perform the reduction to a Bessel generator with potential term. Applying the Feynman-Kac representation of the Bessel semigroup with the new potential we consider the semigroup and the potential term separately. For the semigroup part we use facts about Bessel functions whereas for the potential part we perform a Khasminski-type argument.

The semigroup and generator corresponding to the differential expression $\tau^{\nu,c}$ will be denoted by $(P_t)_{t\geq 0}$ and $L$.

\textbf{Stepy 1} (reduction of drift to potential): Let $C(x):=\int_1^x c(s)\,ds$ and denote by $\rho^{\nu,c}$ the speed measure of the diffusion associated to $\tau^{\nu,c}$. Define the family of operators $(\tilde P_t)_{t\geq 0}$ on $L^2((0,\infty),\rho^{\nu,c})$ via
\[
\tilde P_t f(x) = e^{-C(x)/2} \left(P_t e^{C(\cdot)/2}f(\cdot)\right)(x).
\]
Since
\[
\int_0^\infty \left|f(x)e^{-C(x)/2}\right|\rho^{\nu,c}(dx) = \int_0^\infty |f(x)|^2\rho^{\nu,0}(x)<\infty,
\]
we have that $\tilde P_t f\in L^2((0,\infty),\rho^{\nu,0})$. We see by a simple calculation that the generator $\tilde L$ of the semigroup $(\tilde P_t)_{t\geq 0}$ acts on smooth functions with compact support via
\[
\tilde L f(x) = \tau^\nu f(x) + \left[\frac{1}{2} (c^2(x)-c'(x))-\frac{(-2\nu+1)c(x)}{x}\right]f(x).
\]
As a result $(\tilde P_t)_{t\geq 0}$ is nothing but the Bessel semigroup corresponding to the index $\nu$ with additional potential. Let us see what we gain from this by inverting the transformation
\begin{equation}\label{e:FK}
\begin{split}
e^{-tL}\ind_{[0,\ve]}(x) &= P_t \ind_{[0,\ve]}(x)\\
&= e^{C(x)/2}\tilde P_t(e^{-C(\cdot)/2}\ind_{[0,\ve]})(x)\\
&= e^{C(x)/2}\E^{x}_\nu\left[e^{-C(X_t)/2}\ind_{[0,\ve]}(X_t)e^{-\int_0^t\frac{1}{2}(c^2(X_s)-c'(X_s)) -\frac{(-2\nu+1)c(X_s)}{X_s}\,ds}, T_0>t\right],
\end{split}
\end{equation}
where for the final step we used the Feynman-Kac representation for the Bessel semigroup with potential  $\left[\frac{1}{2} (c^2(x)-c'(x))-\frac{(-2\nu+1)c(x)}{x}\right]$. The right hand side of \eqref{e:FK} can be split using H{\"o}lder's inequality into the product of the following two terms
\begin{equation}\label{e:Bessel_semi}
e^{C(x)/2}\E^x_\nu\left[\ind_{[0,\ve]}(X_t),T_0>t\right]^{\frac{1}{2}}
\end{equation}
\begin{equation}\label{e:Bessel_pot}
\E^x_\nu\left[e^{-C(X_t)/2}e^{-\int_0^t(c^2(X_s)-c'(X_s)) -2\frac{(-2\nu+1)c(X_s)}{X_s}\,ds}\ind_{[0,\ve]}(X_t),T_0>t\right]^{\frac{1}{2}}.
\end{equation}
In the following two steps we estimate the two factors \eqref{e:Bessel_semi} and \eqref{e:Bessel_pot} separately, in order to show the square integrability in $x$.

\textbf{Step 2} (estimating the semigroup factor \eqref{e:Bessel_semi}): To estimate this factor we neet to collect some facts on Bessel processes. Let $I_\nu$ denote the modified Bessel function of the first kind, i.e.
\[
I_\nu = \sum_{k=0}^\infty \frac{(x/2)^{\nu+2k}}{k!\Gamma(\nu+k+1)}
\]
and let $R^\nu$ denote the Bessel process. For $\nu\geq 0$ the process $R^\nu$ has the transition function
\[
p^\nu(t,x,y) = \frac{1}{2t}(xy)^{-\nu}e^{-\frac{x^2+y^2}{2t}}I_\nu\left(\frac{xy}{t}\right)
\]
with respect to the measure $m(dy)=2y^{2\nu+1}dy$. From the $h$-transform property (see page 75 of \cite{BS12}) we get that the transition function for the process $R^{-\nu}$ is given by
\begin{equation}\label{e:R-nu}
\begin{split}
p^{-\nu}(t,x,y) &= \frac{1}{h(x)} p^\nu(t,x,y)h(y) = x^{2\nu}\frac{1}{2t}(xy)^{-\nu}e^{-\frac{x^2+y^2}{2t}}I_\nu\left(\frac{xy}{t}\right) y^{-2\nu}\\
&=\frac{1}{2t} x^\nu e^{-\frac{x^2+y^2}{2t}}I_\nu\left(\frac{xy}{t}\right) y^{-3\nu}
\end{split}
\end{equation}
where $h(x):=\frac{1}{2\nu}x^{-2\nu}$. Note that since we need $0$ to be an exit boundary, we are only interested in $R^{-\nu}$ with $\nu\geq 1$. For the estimates we need some more basic facts on Bessel functions. Observe that for $y\in (0,1)$:
\[
I_\nu\left(\frac{xy}{t}\right)=\sum_{k=0}^\infty\frac{\left(\frac{xy}{2t}\right)^{\nu+2k}}{k!\Gamma(\nu+k+1)}\leq y^\nu\sum_{k=0}^\infty\frac{\left(\frac{x}{2t}\right)^{\nu+2k}}{k!\Gamma(\nu+k+1)} = y^\nu I_\nu(x/t).
\]
We can then obtain (recalling that $\nu$ is negative) an upper bound for \eqref{e:Bessel_semi}
\begin{equation*}
\begin{split}
e^{C(x)/2}\E^x_\nu\left[\ind_{[0,\ve]}(X_t),T_0>t\right]&= e^{C(x)/2}\int_0^\ve p^\nu(t,x,y)\,m(dy)\\
&=e^{C(x)/2}\int_0^\ve \frac{1}{2t}x^\nu e^{-\frac{x^2+y^2}{2t}}I_\nu\left(\frac{xy}{t}\right)y^{-3\nu}y^{2\nu+1}\,dy\\
&\leq e^{C(x)/2} \frac{1}{2t}x^\nu e^{-\frac{x^2}{2t}}I_\nu(x/t)\int_0^\ve e^{-\frac{y^2}{2t}}y\,dy.
\end{split}
\end{equation*}
As such, it suffices to show that $e^{C(x)/2} \left(x^\nu e^{-\frac{x^2}{2t}}I_\nu(x/t)\right)^{\frac{1}{2}}\in L^2((0,\infty),\rho^{-\nu,c})$. This follows from the asymptotics
\[
I_\nu(z) \sim \frac{1}{\Gamma(\nu+1)}\left(\frac{z}{2}\right)^\nu ~\text{for small}~z
\]
and
\[
I_\nu(z) \sim \frac{e^z}{\sqrt{2\pi z}} ~\text{for large}~z.
\]

Step 3 (estimating the Bessel potential factor \eqref{e:Bessel_pot}): As \eqref{e:Bessel_semi} is square integrable it suffices to show that \eqref{e:Bessel_pot} is bounded in $x$. By the assumptions on $c$, it suffices to show that
\[
\sup_{x\in (0,\infty)}\E^x_\nu\left[ e^{ 2\int_0^t\ind_{[0,1/2]}(X_s)\frac{(-2\nu+1)c(X_s)}{X_s}\,ds}\ind_{[0,\ve]}(X_t),T_0>t\right]<\infty.
\]
By the explicit expression for the heat kernel $p^\nu(t,x,y)$ for negative $\nu$ we see that
\begin{equation}\label{e:lim_semi}
\lim_{t\to 0}\sup_{x\in (0,\infty)} \E^x_\nu [f(X_t),T_0>t] = \lim_{t\to 0}\sup_{x\in (0,\infty)} \int_0^\infty p^\nu(t,x,y) f(y)\,m(dy)=0,
\end{equation}
where the function $f:(0,\infty)\to (0,\infty)$ is defined by
\[
f(y) = C\ind_{(0,\ve)}(y)y^{-1},
\]
Observe that \eqref{e:lim_semi} gives for $t>0$ small enough and some $0<a<1$
\begin{equation*}
\begin{split}
\E^x_\nu\left[e^{\int_0^t f(X_s)\,ds}, T_0>t\right] &= \sum_{n\geq 0} \frac{1}{n!} \E^x_\nu \left[\left(\int_0^t f(X_s)\,ds\right)^n,T_0>t\right]\leq \sum_{n\geq 0} a^n<\infty.
\end{split}
\end{equation*}
This implies that for small $t>0$
\[
\sup_{x\in (0,\infty)}\E^x_\nu\left[ e^{ 2\int_0^t\ind_{[0,1/2]}(X_s)\frac{(-2\nu+1)c(X_s)}{X_s}\,ds}\ind_{[0,\ve]}(X_t),T_0>t\right]<\infty.
\]
By the Markov property the above remains true for every $t>0$. This completes the proof of the first assertion from Assumption \ref{a:1}. The $L^2$ convergence of $e^{-tL}\ind_{[\ve,z]}$ follows from similar computations which are left to the reader.
\end{proof}
\subsection{Entrance boundary at $\infty$}\label{s:entrance} We are now ready to prove the main result for entrance boundaries at $\infty$.

\entrance*

The strategy used is similar to the proof of Theorem \ref{t:exit_natural}. Things are simplified because by the spectral representation of Theorem \ref{t:exit_spectral} we obtain
\begin{equation*}
\begin{split}
\lim_{t\to \infty}\frac{e^{-tL}\ind_A(x)}{e^{-tL}\ind_{(0,\infty)}(x)} &= \lim_{t\to \infty} \frac{e^{\lambda_0 t} \int_A p(t,x,y)\rho(dy)}{e^{\lambda_0 t} \int_0^\infty p(t,x,y)\rho(dy)}\\
&=  \lim_{t\to \infty} \frac{\int_A \sum_{k}e^{(\lambda_0-\lambda_k) t}u_{\lambda_k}(x)u_{\lambda_k}(y) \rho(dy)}{\int_0^\infty \sum_{k}e^{(\lambda_0-\lambda_k) t}u_{\lambda_k}(x)u_{\lambda_k}(y) \rho(dy)}
\end{split}
\end{equation*}
which by the order of the eigenvalues already proves the theorem if we are allowed to interchange the limit and the integration. We make use of the parabolic Harnack inequality, which in this setting seems to go back to \cite{C09}.
\begin{proof}
We have seen in Theorem \ref{t:exit_spectral} that for $\infty$ being an entrance boundary the spectrum of $L$ is purely discrete and the lowest eigenfunction is integrable with respect to $\rho$. Moreover, it follows from the spectral representation of Theorem \ref{t:exit_spectral} that
\begin{equation}\label{e:lim_transition}
\lim_{t\to \infty} e^{\lambda_0 t} p(t,x,y) = cu_{\lambda_0}(x)u_{\lambda_0}(y),
\end{equation}
where $c$ is a normalizing and $u_{\lambda_0}$ is the unique positive eigenfunction, normalized to $\|u_{\lambda_0}\|_{L^2((0,\infty),\rho)}=1$, corresponding to the lowest eigenvalue $\lambda_0$. Using the parabolic Harnack principle as in \cite{C09} (see the proof of their Lemma 5.3) and \cite{KS12} we get for some locally bounded function $\Theta_0:(0,\infty)\to (0,\infty)$, every $z\in (0,\infty)$ with $|z-x|<\delta(x)=\frac{1}{2}\wedge \frac{x}{4}$ and every $y\in (0,\infty)$ that
\[
p(t,x,y)\leq \Theta_0 (x)  p(t+1,z,y).
\]
This implies
\begin{equation}\label{e:Theta}
\begin{split}
p(t,x,y) & = \frac{\int_{|z-x|<\delta(x)}p(t,z,y)u_{\lambda_0}(z)\rho(dz)}{\int_{|z-x|<\delta(x)}u_{\lambda_0}(z)\rho(dz)}\\
&\leq \Theta_0(x) \frac{\int_{|z-x|<\delta(x)}p(t+1,z,y)u_{\lambda_0}(z)\rho(dz)}{\int_{|z-x|<\delta(x)}u_{\lambda_0}(z)\rho(dz)}\\
&\leq \Theta_0(x) \frac{\int p(t+1,z,y)u_{\lambda_0}(z)\rho(dz)}{\int_{|z-x|<\delta(x)}u_{\lambda_0}(z)\rho(dz)}\\
&\leq \frac{\Theta_0(x)}{\int_{|z-x|<\delta(x)}u_{\lambda_0}(z)\rho(dz)} e^{-\lambda_0(t+1)}u_{\lambda_0}(y)
\end{split}
\end{equation}
which is integrable as $u_{\lambda_0}$ is integrable. We translate the probabilistic object of interest to the above. For any compactly supported initial distribution $\nu$ one has
\begin{equation*}
\begin{split}
\lim_{t\to \infty} \Pr^\nu(X_t\in A~|~T_0>t) &= \lim_{t\to \infty}\frac{ \Pr^\nu(X_t\in A, T_0>t)}{\Pr^\nu(X_t\in (0,\infty), T_0>t)}\\
&=\lim_{t\to \infty} \frac{\left\langle e^{-tL}\ind_A(\cdot),\nu\right\rangle}{\left\langle e^{-tL}\ind_{(0,\infty)}A(\cdot),\nu\right\rangle}\\
&= \lim_{t\to \infty}\frac{e ^{\lambda_0 t}\left\langle \int_Ap(t,\cdot,y)\rho(dy),\nu\right\rangle}{e ^{\lambda_0 t}\left\langle \int_0^\infty p(t,\cdot,y)\rho(dy),\nu\right\rangle}
\end{split}
\end{equation*}
Taking account of the above bounds we can use dominated convergence to obtain
\begin{equation*}
\begin{split}
\lim_{t\to \infty} \Pr^\nu(X_t\in A~|~T_0>t) &= \frac{\int_0^\infty \int_A \lim_{t\to\infty}e^{\lambda_0 t}p(t,x,y)\nu(dx)\rho(dy)}{\int_0^\infty \int_0^\infty \lim_{t\to\infty}e^{\lambda_0 t}p(t,x,y)\nu(dx)\rho(dy)}\\
&= \frac{\int_0^\infty \int_A u_{\lambda_0}(x)u_{\lambda_0}(y)\nu(dx)\rho(dy)}{\int_0^\infty \int_0^\infty u_{\lambda_0}(x)u_{\lambda_0}(y)\nu(dx)\rho(dy)}\\
&= \frac{\int_A u_{\lambda_0}(y)\rho(dy)}{\int_0^\infty u_{\lambda_0}(y)\rho(dy)}
\end{split}
\end{equation*}
which completes the proof.

The above proof made use of the fact that $\nu$ is compactly supported. One can extend this to any initial distribution $\nu_1$ on $(0,\infty)$ by using the methods from \cite{C09}[Proposition 7.7]. As a result the quasistationary distribution is unique and attracts all initial distributions $\nu_1$ with support on $(0,\infty)$
\end{proof}

\subsection{An application to a diffusion coming from population dynamics}\label{s:bio}
We apply Theorem \ref{t:exit_entrance} to study a process arising in population dynamics.
Consider the SDE
\begin{equation}\label{e:N}
dN_t =(\mu N_t-\kappa N_t^2)\,dt + \sigma N_t\,dW_t+\sqrt{\gamma N_t}\,dB_t
\end{equation}
where $(W_t)_{t\geq 0}$ and $(B_t)_{t\geq 0}$ are independent Brownian motions, $\mu, \kappa, \sigma, \gamma\in (0,\infty)$ and $N_0>0$. This process comes up naturally as a scaling limit of branching diffusions in a random environment. See \cite{H11}, \cite{BH12}, \cite{A05} for more details.

The quadratic variation of the process $(N_t)_{t\geq 0}$ is
\begin{equation*}
d[N]_t = (\gamma N_t + \sigma^2 N_t^2)\,dt.
\end{equation*}
As a result, there exists a Brownian motion $(U_t)_{t\geq 0}$ such that
\begin{equation}\label{E:N_U}
dN_t =(\mu N_t-\kappa N_t^2)\,dt + \sqrt{\gamma N_t + \sigma^2 N_t^2} dU_t.
\end{equation}

Our diffusion $(N_t)_{t\geq 0}$ has drift $b(z)=\mu z-\kappa z^2$ and diffusion coefficient $\sigma(z)=\sqrt{\gamma z+\sigma^2z^2}$. If $a,c>0$ we can find the scale function
\begin{eqnarray*}
s(x)&=&\int_c^x \exp\left(-\int_a^y \frac{2b(z)}{\sigma^2(z)}\,dz\right)\,dy\\
&=& \int_c^x \exp\left(-\int_a^y \frac{2(\mu z-\kappa z^2)}{\gamma z+\sigma^2z^2}\,dz\right)\,dy\\
&=& \int_c^x \left(\frac{\gamma+\sigma^2 y}{\gamma+\sigma^2a}\right)^{-2\left(\frac{\mu}{\sigma^2}+\frac{\gamma\kappa}{\sigma^4}\right)}e^{\frac{2\kappa}{\sigma^2}(y-a)}\,dy.
\end{eqnarray*}
and the density of the speed measure
\begin{eqnarray*}
m(z) &=& \frac{1}{\sigma^2(z) s'(z)} \\
&=& \left(\frac{\gamma+\sigma^2 z}{\gamma+\sigma^2a}\right)^{2\left(\frac{\mu}{\sigma^2}+\frac{\gamma\kappa}{\sigma^4}\right)}e^{-\frac{2\kappa}{\sigma^2}(z-a)}\frac{1}{\gamma z+\sigma^2 z^2}.
\end{eqnarray*}
We make use of \cite{KT81} to classify the boundary points $0$ and $\infty$.
\begin{proposition}\label{p:pop_bound}
Consider the diffusion given by \eqref{E:N_U}. The point $0$ is an exit boundary and the point $\infty$ is an entrance boundary.
\end{proposition}
\begin{proof}
Define
\begin{equation*}
M(0,x]:=\lim_{l\downarrow 0} \int_l^x m(z)\,dz
\end{equation*}
and
\begin{equation*}
\Sigma(0) := \int_0^x\left(\int_0^\xi s'(\eta)\,d\eta\right)m(\xi)\,d\xi.
\end{equation*}
Note that for small $z>0$
\begin{equation*}
m(z)\sim \frac{1}{\gamma z + \sigma^2 z^2}
\end{equation*}
and $s'(z)$ is bounded. It is therefore easy to see that
\begin{eqnarray*}
M(0,x]&\sim&\lim_{l\downarrow 0} \int_l^x \frac{1}{\gamma z + \sigma^2 z^2}\,dz\\
&\sim& \lim_{l\downarrow 0} \ln\frac{z}{\gamma+\sigma^2 z} \big |_l^x\\
&\sim& \lim_{l\downarrow 0} \left(- \ln\frac{l}{\gamma+\sigma^2 l}\right)\\
&=& \infty
\end{eqnarray*}
and
\begin{eqnarray*}
\Sigma(0)&=& \int_0^x \left(\int_0^\xi  \left(\frac{\gamma+\sigma^2 \eta}{\gamma+\sigma^2a}\right)^{-2\left(\frac{\mu}{\sigma^2}+\frac{\gamma\kappa}{\sigma^4}\right)}e^{\frac{2\kappa}{\sigma^2}(\eta-a)}  \,d\eta\right)m(\xi)   \,d\xi\\
&\sim& \int_0^x \xi m(\xi)   \,d\xi\\
&\sim& \int_0^x \xi \frac{1}{\gamma \xi + \sigma^2 \xi^2}\,d\xi\\
&=& \int_0^x\frac{1}{\gamma  + \sigma^2 \xi}\,d\xi\\
&<&\infty.
\end{eqnarray*}
According to Table 6.2 in \cite{KT81} the above imply that $0$ is an exit boundary.

Let us look at the boundary at $\infty$.
To simplify notation, without loss of generality, for any $\xi>a$ we can write
\begin{eqnarray*}
s'(\xi)&=& (1+\xi)^{-C} e^{D\xi}\\
m(\xi) &=& \frac{1}{\gamma \xi + \sigma^2 \xi} (1+\xi)^{C} e^{-D\xi}
\end{eqnarray*}
for some $D,C>0$.
Then, for some fixed large $x\in (0,\infty)$ with $x>a$ compute
\begin{eqnarray*}
\bar N(\infty)&:=& \int_x^\infty\left(\int_\xi^\infty m(\xi)\,d\xi\right)s'(\eta)\,d\eta\\
&=& \int_x^\infty \left(\int_\eta^\infty \frac{1}{\gamma \xi + \sigma^2 \xi^2} (1+\xi)^{C} e^{-D\xi} \,d\xi\right) (1+\eta)^{-C} e^{D\eta}\,d\eta\\
&\leq& \int_x^\infty\left(\int_\eta^\infty  (1+\xi)^{C} e^{-D\xi} \,d\xi\right) \frac{1}{\gamma \eta + \sigma^2 \eta^2}(1+\eta)^{-C} e^{D\eta}\,d\eta
\end{eqnarray*}
Using integration by parts
\begin{equation*}
\int_\eta^\infty  (1+\xi)^{C} e^{-D\xi} \,d\xi = \frac{1}{D} e^{-D\eta} (1+\eta)^C + \int_\eta^\infty \frac{1}{D} e^{-D\xi} C(1+\xi)^{-C-1}\,d\xi
\end{equation*}
and therefore
\begin{equation*}
\int_\eta^\infty  (1+\xi)^{C} e^{-D\xi} \,d\xi \simeq \frac{1}{D} e^{-D\eta} (1+\eta)^C  ~\text{for}~\eta\rightarrow\infty.
\end{equation*}
As a result we see that
\begin{eqnarray*}
\bar N(\infty)&\simeq& \int_x^\infty\left(\frac{1}{D} e^{-D\eta} (1+\eta)^C\right) \frac{1}{\gamma \eta + \sigma^2 \eta^2}(1+\eta)^{-C} e^{D\eta}\,d\eta\\
&\simeq& \int_x^\infty \frac{1}{\gamma \eta + \sigma^2 \eta^2}\,d\eta\\
&<&\infty.
\end{eqnarray*}
One also clearly has
\begin{eqnarray*}
s(\infty) &=& \int_c^\infty \left(\frac{\gamma+\sigma^2 y}{\gamma+\sigma^2a}\right)^{-2\left(\frac{\mu}{\sigma^2}+\frac{\gamma\kappa}{\sigma^4}\right)}e^{\frac{2\kappa}{\sigma^2}(y-a)}\,dy\\
&=&\infty.
\end{eqnarray*}
Making use of Table 6.2 from \cite{KT81} we conclude that $\infty$ is an entrance boundary.

\end{proof}
Combining Proposition \ref{p:pop_bound} and Theorem \ref{t:exit_entrance} one gets the following result.
\begin{corollary}
Consider the diffusion given by \eqref{e:N}. Then $(N_t)_{t\geq 0}$ converges to its unique quasistationary distribution.
\end{corollary}
\begin{remark}
We note that the methods from \cite{C09} are not easy to use for the diffusion from \eqref{e:N} because one needs to check extra assumptions. In our case having that $0$ is exit and $\infty$ is entrance is enough to give us the desired result.
\end{remark}

{\bf Acknowledgments.}  We thank Nicolas Champagnat, Leif D\"{o}ring, Alison Etheridge, Steve Evans, Fritz Gesztesy, Sebastian Schreiber and Denis Villemonais for very insightful comments and suggestions.

\bibliographystyle{amsalpha}
\bibliography{qsd}

\newcommand{\etalchar}[1]{$^{#1}$}
\providecommand{\bysame}{\leavevmode\hbox to3em{\hrulefill}\thinspace}
\providecommand{\MR}{\relax\ifhmode\unskip\space\fi MR }
\providecommand{\MRhref}[2]{%
  \href{http://www.ams.org/mathscinet-getitem?mr=#1}{#2}
}
\providecommand{\href}[2]{#2}
\begin{thebibliography}{CMSM95}

\bibitem[AG54]{AG54}
N.~I. Achieser and I.~M. Glasmann, \emph{Theorie der {L}inearen {O}peratoren im
  {H}ilbert-{R}aum}, vol. 1054, Berlin Akademie-Verlag, 1954.

\bibitem[BC15]{BC15}
M.~Benaim and B.~Cloez, \emph{A stochastic approximation approach to
  quasi-stationary distributions on finite spaces}, Electron. Commun. Probab.
  \textbf{20} (2015), 13 pp.

\bibitem[BH12]{BH12}
C.~Boeinghoff and M.~Hutzenthaler, \emph{Branching diffusions in random
  environment}, Markov Processes and Related Fields \textbf{18} (2012), no.~2.

\bibitem[BS12]{BS12}
A.~N. Borodin and P.~Salminen, \emph{Handbook of brownian motion-facts and
  formulae}, Birkh{\"a}user, 2012.

\bibitem[CCL{\etalchar{+}}09]{C09}
P.~Cattiaux, P.~Collet, A.~Lambert, S.~Mart{\'\i}nez, S.~M{\'e}l{\'e}ard, and
  J.~San Mart{\'\i}n, \emph{Quasi-stationary distributions and diffusion models
  in population dynamics}, The Annals of Probability (2009), 1926--1969.

\bibitem[CCM16]{CCM16}
J.-R. Chazottes, P.~Collet, and S.~M{\'e}l{\'e}ard, \emph{Sharp asymptotics for
  the quasi-stationary distribution of birth-and-death processes}, Probability
  Theory and Related Fields \textbf{164} (2016), no.~1, 285--332.

\bibitem[CCPV16]{CCV16}
N.~Champagnat, A.~Coulibaly-Pasquier, and D.~Villemonais, \emph{Exponential
  convergence to quasi-stationary distribution for multi-dimensional diffusion
  processes}, arXiv preprint arXiv:1603.07909 (2016).

\bibitem[CMSM95]{CMSM95}
P.~Collet, S.~Mart{\'\i}nez, and J.~San~Mart{\'\i}n, \emph{Asymptotic laws for
  one-dimensional diffusions conditioned to nonabsorption}, The Annals of
  Probability (1995), 1300--1314.

\bibitem[CV15a]{CV152}
N.~Champagnat and D.~Villemonais, \emph{Exponential convergence to
  quasi-stationary distribution for absorbed one-dimensional diffusions with
  killing}, arXiv preprint arXiv:1510.05794 (2015).

\bibitem[CV15b]{CV15}
\bysame, \emph{Uniform convergence of conditional distributions for absorbed
  one-dimensional diffusions}, arXiv preprint arXiv:1506.02385 (2015).

\bibitem[CV16]{CV14}
\bysame, \emph{Exponential convergence to quasi-stationary distribution and
  {$Q$}-process}, Probability Theory and Related Fields \textbf{164} (2016),
  no.~1, 243--283.

\bibitem[DM14]{DM14}
P.~Diaconis and L.~Miclo, \emph{On quantitative convergence to
  quasi-stationarity}, arXiv preprint arXiv:1406.1805 (2014).

\bibitem[FOT10]{FOT10}
M.~Fukushima, Y.~Oshima, and M.~Takeda, \emph{Dirichlet forms and symmetric
  {M}arkov processes}, vol.~19, Walter de Gruyter, 2010.

\bibitem[GZ06]{GZ06}
F.~Gesztesy and M.~Zinchenko, \emph{On spectral theory for {S}chr{\"o}dinger
  operators with strongly singular potentials}, Mathematische Nachrichten
  \textbf{279} (2006), no.~9-10, 1041--1082.

\bibitem[Hut11]{H11}
M.~Hutzenthaler, \emph{Supercritical branching diffusions in random
  environment}, Electron. Commun. Probab \textbf{16} (2011), no.~2.

\bibitem[KS12]{KS12}
M.~Kolb and D.~Steinsaltz, \emph{Quasilimiting behavior for one-dimensional
  diffusions with killing}, The Annals of Probability \textbf{40} (2012),
  no.~1, 162--212.

\bibitem[KT81]{KT81}
S.~Karlin and H.~Taylor, \emph{A second course in stochastic processes},
  Academic Press, Orlando, 1981.

\bibitem[Lam05]{A05}
A.~Lambert, \emph{The branching process with logistic growth}, The Annals of
  Applied Probabability \textbf{15} (2005), no.~2, 1506--1535.

\bibitem[Lit12]{L12}
J.~Littin, \emph{Uniqueness of quasistationary distributions and discrete
  spectra when $\infty$ is an entrance boundary and 0 is singular}, Journal of
  Applied Probability \textbf{49} (2012), no.~3, 719--730.

\bibitem[Man61]{M61}
P.~Mandl, \emph{Spectral theory of semi-groups connected with diffusion
  processes and its application}, Czechoslovak Mathematical Journal \textbf{11}
  (1961), no.~4, 558--569.

\bibitem[MSM01]{MSM01}
S.~Mart{\'\i}nez and J.~San~Mart{\'\i}n, \emph{Rates of decay and h-processes
  for one dimensional diffusions conditioned on non-absorption}, Journal of
  Theoretical Probability \textbf{14} (2001), no.~1, 199--212.

\bibitem[Muc72]{M72}
B.~Muckenhoupt, \emph{Hardy's inequality with weights}, Studia Mathematica
  \textbf{1} (1972), no.~44, 31--38.

\bibitem[MV12]{MV12}
S.~M{\'e}l{\'e}ard and D.~Villemonais, \emph{Quasi-stationary distributions and
  population processes}, Probability Surveys \textbf{9} (2012), 340--410.

\bibitem[Pin95]{P95}
R.~G. Pinsky, \emph{Positive harmonic functions and diffusion}, vol.~45,
  Cambridge university press, 1995.

\bibitem[Pin09]{P09}
\bysame, \emph{Explicit and almost explicit spectral calculations for diffusion
  operators}, Journal of Functional Analysis \textbf{256} (2009), no.~10,
  3279--3312.

\bibitem[SE04]{ES04}
D.~Steinsaltz and S.~N. Evans, \emph{Markov mortality models: implications of
  quasistationarity and varying initial distributions}, Theoretical Population
  Biology \textbf{65} (2004), no.~4, 319--337.

\bibitem[SE07]{ES07}
D.~Steinsaltz and S.~Evans, \emph{Quasistationary distributions for
  one-dimensional diffusions with killing}, Transactions of the American
  Mathematical Society \textbf{359} (2007), no.~3, 1285--1324.

\bibitem[Sim93]{S93}
B.~Simon, \emph{Large time behavior of the heat kernel: On a theorem of
  {C}havel and {K}arp}, Proceedings of the American Mathematical Society
  \textbf{118} (1993), no.~2, 513--514.

\bibitem[Wei67]{W67}
J.~Weidmann, \emph{Zur {S}pektraltheorie von {S}turm-{L}iouville-{O}peratoren},
  Mathematische Zeitschrift \textbf{98} (1967), no.~4, 268--302.

\bibitem[Wei00]{WI}
\bysame, \emph{{L}ineare {O}peratoren in {H}ilbertr{\"a}umen: {T}eil 1
  {G}rundlagen}, Springer-Verlag, 2000.

\bibitem[Wei03]{WII}
\bysame, \emph{{L}ineare {O}peratoren in {H}ilbertr{\"a}umen: {T}eil 2
  {A}nwendungen}, Springer-Verlag, 2003.

\bibitem[Wie85]{W85l}
N.~Wielens, \emph{The essential self-adjointness of generalized schr{\"o}dinger
  operators}, Journal of functional analysis \textbf{61} (1985), no.~1,
  98--115.

\end{thebibliography}

\end{document}